\documentclass[12pt]{amsbook}
\usepackage{graphicx, verbatim,amsmath,amssymb, amscd}

\newcommand{\be}{\begin{equation}}
\newcommand{\ee}{\end{equation}}

\newcommand{\T}{{\mathcal T}}
\newcommand{\R}{{\mathbb R}}

\newcommand{\Z}{{\mathbb Z}}

\newtheorem{thm}{Theorem}[section]
\newtheorem*{thm*}{Theorem}
\newtheorem{cor}[thm]{Corollary}

\everymath{\displaystyle}

\theoremstyle{definition}

\begin{document}
\title{Notes on Differential Forms}
\author{Lorenzo Sadun}
\address{Department of Mathematics\\The University of
  Texas at Austin\\ Austin, TX 78712} 

\date{\today}

\maketitle

\setlength{\baselineskip}{.6cm}

\chapter{Forms on $\R^n$}
This is a series of lecture notes, with embedded problems,
aimed at students studying differential topology. 

Many revered texts, such as Spivak's {\em Calculus on Manifolds} and 
Guillemin and Pollack's {\em Differential Topology} introduce forms 
by first working through properties of alternating tensors. Unfortunately,
many students get bogged down with the whole notion of tensors
and never get to the punch lines: Stokes' Theorem, de Rham cohomology,
Poincare duality, and the realization of various topological invariants
(e.g. the degree of a map) via forms, none of which actually require tensors
to make sense! 

In these notes, we'll follow a different approach, following the philosophy
of Amy's Ice Cream: 

\centerline{\em Life is uncertain. Eat dessert first.}

We're first going to 
define forms on $\R^n$ via unmotivated formulas, develop some proficiency
with calculation, show that forms behave
nicely under changes of coordinates, and prove Stokes' Theorem. 
This approach has the disadvantage that it 
doesn't develop deep intuition, but the strong advantage that the key
properties of forms emerge quickly and cleanly. 

Only {\bf then}, in Chapter 3, will we go back and 
show that tensors
with certain (anti)symmetry properties have the exact same properties as
the formal objects that we studied in Chapters 1 and 2. This allows us to
re-interpret all of our old results from a more modern perspective, and
move onwards to using forms to do topology. 

\section{What is a form?}

On $\R^n$, we start with the symbols $dx^1, \ldots, dx^n$, which at this point
are pretty much meaningless. We define a multiplication operation on
these symbols, denoted by a $\wedge$, subject to the condition
$$ dx^i \wedge dx^j = - dx^j \wedge dx^i.$$
Of course, we also want the usual properties of multiplication to also hold.
If $\alpha, \beta, \gamma$ are arbitrary products of $dx^i$'s, and if $c$
is any constant, then
\begin{eqnarray}
(\alpha + \beta) \wedge \gamma & = & \alpha \wedge \gamma 
+ \beta \wedge \gamma\cr
\alpha \wedge (\beta+\gamma) & = & \alpha \wedge \beta + \alpha \wedge \gamma \cr
(\alpha \wedge \beta) \wedge \gamma & = & \alpha \wedge (\beta \wedge \gamma) \cr
(c \alpha) \wedge \beta & = & \alpha \wedge (c \beta) = c(\alpha \wedge \beta)
\end{eqnarray}

Note that the anti-symmetry implies that $dx^i \wedge dx^i = 0$. 
Likewise, if $I = \{i_1,
\ldots, i_k\}$ is a list of indices where some index gets repeated, 
then $dx^{i_1}\wedge \cdots \wedge dx^{i_k}=0$, since we can swap the order
of terms (while keeping track of signs) until the same index comes up twice
in a row. For instance, 
$$ dx^1 \wedge dx^2 \wedge dx^1 = - dx^1 \wedge dx^1 \wedge dx^2 = 
-(dx^1 \wedge dx^1) \wedge dx^2 = 0.$$
  
\begin{itemize}
\item
A {\em 0-form} on $\R^n$ is just a function.
\item A {\em 1-form} is an expression of the form $\sum_i f_i(x) dx^i$, 
where $f_i(x)$ is a function and $dx^i$ is one of our meaningless symbols.
\item A {\em 2-form} is an expression of the form $\sum_{i,j} f_{ij}(x) dx^i
\wedge dx^j$. 
\item A {\em $k$-form} is an expression of the form $\sum_I f_I(x) dx^I$,
where $I$ is a subset $\{i_1,\ldots,i_k\}$ of $\{1,2,\ldots,n\}$ and 
$dx^I$ is shorthand for $dx^{i_1} \wedge \cdots \wedge dx^{i_k}$. 
\item If $\alpha$ is a $k$-form, we say that $\alpha$ has {\em degree $k$}. 
\end{itemize}

For instance, on $\R^3$
\begin{itemize}
\item 0-forms are functions
\item 1-forms look like $P dx + Q dy + R dz$, where $P$, $Q$ and $R$ are
functions and we are writing $dx, dy, dz$ for $dx^1$, $dx^2$, $dx^3$. 
\item 2-forms look like $P dx\wedge dy + Q dx \wedge dz + R dy \wedge dz$. 
Or we could just as well write $P dx \wedge dy - Q dz \wedge dx + R
dy \wedge dz$. 
\item 3-forms look like $f dx \wedge dy \wedge dz$. 
\item There are no (nonzero) forms of degree greater than 3.
\end{itemize}

When working on $\R^n$, there are exactly $n \choose k$ linearly independent 
$dx^I$'s of degree $k$, and $2^n$ linearly independent $dx^I$'s in all 
(where we include $1=dx^I$ when $I$ is the empty list). If $I'$ is a permutation
of $I$, then $dx^{I'} = \pm dx^I$, and it's silly to include both 
$f_I dx^I$ and $f_{I'} dx^{I'}$ in our expansion of a $k$-form. Instead, one
usually picks a preferred ordering of $\{i_1,\ldots, i_k\}$ (typically 
$i_1 < i_2 < \cdots < i_k$) and restrict our sum to $I$'s of that sort. When
working with 2-forms on $\R^3$, we can use $dx \wedge dz$ or $dz \wedge dz$,
but we don't need both. 


If $\alpha = \sum \alpha_I(x) dx^I$ is a $k$-form 
and $\beta = \sum \beta_J(x) dx^J$ is an $\ell$-form, then we define
$$ \alpha \wedge \beta = \sum_{I,J} \alpha_I(x)\beta_J(x) dx^I \wedge dx^J.$$
Of course, if $I$ and $J$ intersect, then $dx^I \wedge dx^J=0$. Since 
going from $(I,J)$ to $(J,I)$ involves $k\ell$ swaps, we have 
$$dx^J \wedge dx^I = (-1)^{k\ell} dx^I \wedge dx^J, $$
and likewise $\beta \wedge \alpha = (-1)^{k\ell} \alpha \wedge \beta$. 
Note that the wedge product of a 0-form (aka function) with a $k$-form is 
just ordinary multiplication. 

\section{Derivatives of forms}

If $\alpha = \sum_I \alpha_I dx^I$ is a $k$-form, then we define
the {\em exterior derivative}
$$ d\alpha = \sum_{I,j} \frac{\partial \alpha_I(x)}{\partial x^j}
dx^j \wedge dx^I.$$
Note that $j$ is a single index, not a multi-index. For instance, on $\R^2$, 
if $\alpha = xy dx + e^x dy$, then 
\begin{eqnarray} d \alpha & = & y dx \wedge dx + x dy \wedge dx 
+ e^x dx \wedge dy + 0 dy \wedge dy \cr 
& = & (e^x-x) dx \wedge dy.
\end{eqnarray}
If $f$ is a 0-form, then we have something even simpler:
$$ df(x) = \sum \frac{\partial f(x)}{\partial x^j} dx^j,$$
which should look familiar, if only as an imprecise calculus formula. 
One of our goals is to make such statements precise and rigorous. 
Also, remember that $x^i$ is actually a function on $\R^n$. Since 
$\partial_j x^i = 1$ if $i=j$ and 0 otherwise, $d(x^i) = dx^i$, which suggests that
our formalism isn't totally nuts. 

The key properties of the exterior derivative operator $d$ are listed in
the following
\begin{thm}
\begin{enumerate}
\item
If $\alpha$ is a $k$-form and $\beta$ is an $\ell$-form, then 
$$ d(\alpha \wedge \beta) = (d \alpha) \wedge \beta + (-1)^k \alpha \wedge
(d \beta).$$
\item $d(d\alpha)=0$. (We abbreviate this by writing $d^2=0$.) 
\end{enumerate}
\end{thm}

\begin{proof} For simplicity, we prove this for the case where
 $\alpha = \alpha_I dx^I$ and $\beta = \beta_J dx^J$ each have only a single
term. The general case then follows from linearity. 

The first property is essentially the product rule for 
derivatives. 
\begin{eqnarray} 
\alpha \wedge \beta & = & \alpha_I(x) \beta_J(x) dx^I \wedge dx^J \cr 
d(\alpha \wedge \beta) & = & \sum_j \partial_j(\alpha_I(x) \beta_J(x))
dx^j \wedge dx^I \wedge dx^J \cr 
& = & \sum_j (\partial_j \alpha_I(x)) \beta_J(x) dx^j \wedge dx^I \wedge dx^J
\cr && + \sum_j \alpha_I(x) \partial_j \beta_J(x) dx^j \wedge dx^I \wedge dx^J
\cr & = & \sum_j (\partial_j \alpha_I(x)) dx^j \wedge dx^I \wedge 
\beta_J(x) dx^J
\cr && + (-1)^k \sum_j \alpha_I(x) dx^I \wedge \partial_j \beta_J(x) dx^j \wedge
dx^J
\cr &=& (d\alpha) \wedge \beta + (-1)^k \alpha \wedge d\beta.
\end{eqnarray}

The second property for 0-forms (aka functions) is just ``mixed partials are
equal'': 
\begin{eqnarray} d(df) & = &  d (\sum_i \partial_if dx^i) \cr 
& = & \sum_j\sum_i \partial_j\partial_i f dx^j \wedge dx^i\cr 
& = & - \sum_{i,j} \partial_i\partial_j f dx^i \wedge dx^j \cr  
& = & - d(df) = 0,
\end{eqnarray}
where in the third line we used 
$\partial_j\partial_i f = \partial_i \partial_j f$ and 
$dx^i \wedge dx^j = - dx^j \wedge dx^i$. We then use the first property, and
the (obvious) fact that $d(dx^I)=0$, to extend this to $k$-forms:
\begin{eqnarray}
d(d\alpha) & = & d (d\alpha_I \wedge dx^I) \cr 
& = & (d (d \alpha_I))\wedge dx^I - d\alpha_I \wedge d(dx^I) \cr 
& = & 0 - 0 = 0.
\end{eqnarray}
where in the second line we used the fact that $d\alpha_I$ is a 1-form,
and in the third line used the fact that $d(d\alpha_I)$ is $d^2$ applied to
a function, while $d(dx^I)=0$.
\end{proof}

\smallskip

\noindent {\bf Exercise 1:} On $\R^3$, there are interesting 1-forms 
and 2-forms associated with each vector field 
$\vec v(x) = (v_1(x), v_2(x), v_3(x))$. 
(Here $v_i$ is a component of the vector $\vec v$, not a vector in its
own right.) Let $\omega^1_{\vec v} = v_1dx + v_2dy + v_3dz$, and let 
$\omega^2_{\vec v}=
v_1dy \wedge dz + v_2dz \wedge dx + v_3dx \wedge dy$.
Let f be a function. Show that (a) $df = \omega^1_{\nabla f}$, 
(b) $d\omega^1_{\vec v}  = \omega^2_{\nabla \times \vec v}$, and (c) $d \omega^2_{\vec v} =
(\nabla \cdot \vec v)\, dx\wedge dy \wedge dz$, 
where $\nabla$, $\nabla \times$ and $\nabla \cdot$ 
are the usual gradient, curl, and divergence
operations.

\smallskip

\noindent {\bf Exercise 2:} A form $\omega$ is called {\em closed} 
if $d\omega = 0$, 
and {\em exact} if $\omega = d \nu$ for some other form $\nu$.
Since $d^2 = 0$, all exact forms are closed. On $\R^n$ it happens that
all closed forms of nonzero degree are exact. 
(This is called the Poincare Lemma). However, on
subsets of $\R^n$ the Poincare Lemma does not necessarily hold.  
On $\R^2$ minus the origin, 
show that $\omega = (xdy - ydx)/(x^2+y^2)$
is closed. We will soon see that $\omega$ is not exact.

\section{Pullbacks}

Suppose that $g: X \to Y$ is a smooth map, where $X$ is an open subset
of $\R^n$ and $Y$ is an open subset of $\R^m$, and that $\alpha$ is a
$k$-form on $Y$. We want to define a {\em pullback form} $g^* \alpha$
on $X$.  Note that, as the name implies, the pullback operation
reverses the arrows!  While $g$ maps $X$ to $Y$, and $dg$ maps tangent
vectors on $X$ to tangent vectors on $Y$, $g^*$ maps forms on $Y$ to
forms on $X$.

\begin{thm}
  There is a unique linear map $g^*$ taking forms on $Y$ to forms on
  $X$ such that the following properties hold:
\begin{enumerate}
\item If $f: Y \to \R$ is a function on $Y$, then $g^*f = f \circ g$. 
\item If $\alpha$ and $\beta$ are forms on $Y$, then 
$g^*(\alpha \wedge \beta)=(g^* \alpha) \wedge (g^* \beta).$
\item If $\alpha$ is a form on $Y$, then $g^*(d\alpha) = d(g^*(\alpha))$. (Note that there are really two
different $d$'s in this equation. On the left hand side $d$ maps $k$-forms on $Y$ to $(k+1)$-forms on $Y$. 
On the right hand side, $d$ maps $k$ forms on $X$ to $(k+1)$-forms on $X$. )
\end{enumerate}
\end{thm}

\begin{proof}
The pullback of 
$0$-forms is defined by the first property. However, note that on 
$Y$, the form $dy^i$ is $d$ of the {\em function} $y^i$ (where we're using 
coordinates $\{y^i\}$ on $Y$ and reserving $x$'s for $X$). This means that 
$g^*(dy^i)(x) = d(y^i\circ g)(x) = dg^i(x)$, where $g^i(x)$ is the $i$-th 
component of $g(x)$. But that gives us our formula in general! If $\alpha =
\sum_I \alpha_I(y) dy^I$, then
\be\label{pullback-1} 
g^* \alpha(x) = \sum_I \alpha_I(g(x)) dg^{i_1}\wedge dg^{i_2} \wedge
\cdots \wedge dg^{i_k}.
\ee
Using the formula (\ref{pullback-1}), it's easy to see that 
$g^*(\alpha \wedge \beta) = g^*(\alpha) \wedge g^*(\beta)$.
Checking that $g^*(d\alpha) = d(g^* \alpha)$ in general is left as an
exercise in definition-chasing.
\end{proof}

\smallskip
\noindent{\bf Exercise 3:} Do that exercise!    
\smallskip
         
An extremely important special case is where $m=n=k$. The $n$-form $dy^1 \wedge \cdots \wedge dy^n$ is called 
the {\em volume form} on $\R^n$. 

\smallskip

\noindent {\bf Exercise 4:} Let $g$ is a smooth map from $\R^n$ to
$\R^n$, and let $\omega$ be the volume form on $\R^n$. Show that $g^*
\omega$, evaluated at a point $x$, is $\det(dg_x)$ times the volume
form evaluated at $x$.

\smallskip

\noindent {\bf Exercise 5:} An important property of pullbacks is that
they are {\em natural}. If $g: U \to V$ and $h: V \to W$, where $U$,
$V$, and $W$ are open subsets of Euclidean spaces of various
dimensions, then $h \circ g$ maps $U \to W$. Show that $(h \circ g)^*
= g^* \circ h^*$.

\smallskip

\noindent {\bf Exercise 6:} Let $U = (0,\infty) \times (0, 2\pi)$, and let $V$
be $\R^2$ minus the non-negative $x$ axis. We'll use coordinates $(r,\theta)$
for $U$ and $(x,y)$ for $V$. Let 
$g(r,\theta)= (r \cos(\theta), r \sin(\theta))$, and let $h = g^{-1}$. 
On $V$, let $\alpha = e^{-(x^2+y^2)} dx \wedge dy$. \newline
(a) Compute $g^*(x)$, $g^*(y)$, $g^*(dx)$, $g^*(dy)$, $g^*(dx \wedge dy)$ and 
$g^*\alpha$ (preferably in that order). \newline
(b) Now compute $h^*(r)$, $h^*(\theta)$, $h^*(dr)$ and $h^*(d\theta)$. 

The upshot of this exercise is that pullbacks are something that you have been
doing for a long time! Every time you do a change of coordinates in calculus,
you're actually doing a pullback. 

\section{Integration}

Let $\alpha$ be an $n$-form on $\R^n$, and suppose that $\alpha$ is
compactly supported. (Being compactly supported is overkill, but we're
assuming it to guarantee integrability and to allow manipulations like
Fubini's Theorem. Later on we'll soften the assumption using
partitions of unity.) Then there is only one multi-index that
contributes, namely $I=\{1,2,\ldots,n\}$, and $\alpha(x) = \alpha_I(x)
dx^1 \wedge \cdots \wedge dx^n$. We define \be \label{integral} 
\int_{\R^n} \alpha :=
\int_{\R^n} \alpha_I(x) |dx^1 \cdots dx^n|. \ee The left hand side is
the integral of a form that involves wedges of $dx^i$'s. The right
hand side is an ordinary Riemann integral, in which $|dx^1\cdots dx^n|$
is the usual volume measure (sometimes written $dV$ or $d^nx$). 
Note that the order of the variables in the wedge product,
$x^1$ through $x^n$, is implicitly using the standard orientation of
$\R^n$.  Likewise, we can define the integral of $\alpha$ over any
open subset $U$ of $\R^n$, as long as $\alpha$ restricted to $U$ is
compactly supported.

We have to be a little careful with the left-hand-side of (\ref{integral})
when $n=0$. In this case, $\R^n$ is a single point (with 
positive orientation), and $\alpha$ is just 
a number. We take $\int \alpha$ to be that number. 

\smallskip

\noindent {\bf Exercise 7:} Suppose $g$ is an orientation-preserving
diffeomorphism from an open subset $U$ of $\R^n$ to another open
subset $V$ (either or both of which may be all of $\R^n$). Let
$\alpha$ be a compactly supported $n$-form on $V$. Show that
$$\int_U g^* \alpha = \int_V \alpha.$$
How would this change if $g$ were orientation-reversing? [Hint: use the change-of-variables formula for multi-dimensional 
integrals. Where does the Jacobian come in?]

Now we see what's so great about differential forms! The way they
transform under change-of-coordinates is perfect for defining
integrals. Unfortunately, our development so far only allows us to
integrate $n$-forms over open subsets of $\R^n$. More generally, we'd
like to integrate $k$-forms over $k$-dimensional objects. But this
requires an additional level of abstraction, where we define forms on
manifolds.

Finally, we consider how to integrate something that isn't compactly supported.
If $\alpha$ is not compactly supported, we pick a partition 
of unity $\{\rho_i\}$ such that each $\rho_i$ is compactly supported, and 
define $\int \alpha = \sum \int \rho_i \alpha$. Having this sum be independent
of the choice of partition-of-unity is a question of absolute convergence. 
If $\int_{\R^n} |\alpha_I(x)| dx^1 \cdots dx^n$ converges as a Riemann integral,
then everything goes through. (The proof isn't hard, and is a good exercise
in understanding the definitions.)

\section{Differential forms on manifolds}

An $n$-manifold is a (Hausdorff) space that locally looks like $\R^n$.
We defined abstract smooth $n$-manifolds via structures on the coordinate
charts. If $\psi: U \to X$ is a parametrization of a neighborhood of 
$p \in X$, where $U$ is an open set in $\R^n$, then we associate functions
on $X$ near $p$ with functions on $U$ near $\psi^{-1}(p)$. We associate 
tangent vectors in $X$ with velocities of paths in $U$, or with derivations
of functions on $U$. Likewise, we associated differential forms on $X$ that
are supported in the coordinate neighborhood with differential forms on $U$. 

All of this has to be done ``mod identifications''.  If $\psi_{1,2}: U_{1,2}
\to X$ are parametrizations of the same neighborhood of $X$, then $p$ is 
associated with both $\psi_1^{-1}(p) \in U_1$ and $\psi_2^{-1}(p) \in U_2$.
More generally, if we have an atlas of parametrizations $\psi_i: U_i \to X$,
and if $g_{ij} = \psi_j^{-1} \circ \psi_i$ is the transition function from the
$\psi_i$ coordinates to the $\psi_j$ coordinates on their overlap, then we
constructed $X$ as an abstract manifold as 
\be X = \coprod U_i / \sim, \qquad x \in U_i \sim g_{ij}(x) \in U_j. \ee
We had a similar construction for tangent vectors, and we can do the same
for differential forms. 

Let $\Omega^k(U)$ denote the set of $k$-forms on a subset $U \in \R^n$, and
let $V$ be a coordinate neighborhood of $p$ in $X$. We define
\be \label{forms-on-X} 
\Omega^k(V) = \coprod \Omega^k(U_1)/\sim, \qquad \alpha \in \Omega^k(U_j)
\sim g_{ij}^*(\alpha) \in \Omega^k(U_i). \ee
Note the direction of the arrows. $g_{ij}$ maps $U_i$ to $U_j$, so 
the pullback $g_{ij}^*$ maps forms on $U_j$ to forms on $U_i$. Having defined
forms on neighborhoods, we stitch things together in the usual way. 
A form on $X$ is a collection of forms on
the coordinate neighborhoods of $X$ that agree on their overlaps. 

Let $\nu$ denote a form on $V$, as represented by a form $\alpha$ on $U_j$.
We then write $\alpha = \psi_j^*(\nu)$. As with the polar-cartesian exercise
above, writing a form in a particular set of coordinates is technically
pulling it back to the Euclidean space where those coordinates live. Note
that $\psi_i = \psi_j \circ g_{ij}$, and 
that $\psi_i^* = g_{ij}^* \circ \psi_j^*$, since the realization of $\nu$ in 
$U_i$ is (by equation (\ref{forms-on-X})) the pullback, by $g_{ij}$, 
of the realization of $\nu$ in $U_j$.

This also tells us how to do calculus with forms on manifolds. If 
$\mu$ and $\nu$ are forms on $X$, then
\begin{itemize}
\item The wedge product $\mu \wedge \nu$ is the form whose realization
on $U_i$ is $\psi_i^*(\mu) \wedge \psi_i^*(\nu)$. In other words,
$\psi_i^*(\mu \wedge \nu) = \psi_i^* \mu \wedge \psi_i^* \nu$. 
\item The exterior derivative $d\mu$ is the form whose realization on
$U_i$ is $d(\psi_i^*(\mu))$. In other words, $\psi_i^*(d\mu) = d(\psi_i^* \mu)$.
\end{itemize}

\smallskip

\noindent {\bf Exercise 8:} Show that $\mu \wedge \nu$ and $d\mu$ are 
well-defined.

\smallskip

Now suppose that we have a map $f: X \to Y$ of manifolds and that $\alpha$
is a form on $Y$. The pullback $f^*(\alpha)$ is defined via coordinate patches.
If $\phi: U\subset \R^n \to X$ and 
$\psi: V \subset \R^m \to Y$ are parametrizations of $X$ and $Y$, then there
is a map $h: U \to V$ such that $\psi(h(x)) = f(\phi(x))$. 
We define $f^*(\alpha)$ to be the form of $X$ whose realization in $U$ is 
$h^* \circ (\psi^* \alpha)$. In other words, 
\be \phi^*(f^* \alpha) = h^*(\psi^* \alpha). \ee

An important special case is where $X$ is a submanifold of $Y$ and $f$ is
the inclusion map.  Then $f^*$ is the restriction of $\alpha$ to $X$. When
working with manifolds in $\R^N$, we often write down formulas for $k$-forms
on $\R^N$, and then say ``consider this form on $X$''. E.g., one might say
``consider the 1-form $x dy - y dx$ on the unit circle in $\R^2$''. Strictly
speaking, this really should be ``consider the pullback to $S^1 \subset \R^2$ 
by inclusion of the 1-form $x dy - y dx$ on $\R^2$,'' but (almost) nobody
is {\em that} pedantic! 
  
\section{Integration on oriented manifolds}

Let $X$ be an oriented $k$-manifold, and let $\nu$ be a $k$-form on 
$X$ whose support is a compact subset of a single coordinate chart
$V=\psi_i(U_i)$, where $U_i$ is an open subset of $\R^k$. 
Since $X$ is oriented, we can require that $\psi_i$ be 
orientation-preserving. We then define
\be \int_X \nu = \int_{U_i} \psi_i^* \nu. \ee

\smallskip

\noindent{\bf Exercise 9:} Show that this definition does not depend on the
choice of coordinates. That is, if $\psi_{1,2}: U_{1,2} \to V$ are two sets
of coordinates for $V$, both orientation-preserving, that 
$\int_{U_1} \psi_1^* \nu = \int_{U_2} \psi_2^* \nu$. 

\smallskip

If a form is not supported in a single coordinate chart, we pick an open 
cover of $X$ consisting of coordinate neighborhoods, pick a partition-of-unity
subordinate to that cover, and define
$$ \int_X \nu = \sum \int_X \rho_i \nu. $$
We need a little bit of notation to specify when this makes sense. If 
$\alpha=\alpha_I(x) dx^1\wedge\cdots\wedge dx^k$ is a $k$-form on $\R^k$, 
let $|\alpha| = |\alpha_I(x)| dx^1\wedge\cdots\wedge dx^k$. 
We say that $\nu$ is {\em absolutely integrable} if each
$|\psi_i^* (\rho_i \nu)|$ is integrable over $U_i$, and if the sum of those
integrals converges. It's not hard to show that being absolutely integrable 
with respect to one set of coordinates and partition of unity implies 
absolute integrability with respect to arbitrary coordinates and partitions 
of unity. Those are the conditions under which $\int_X \nu$ unambiguously makes
sense. 

When $X$ is compact and $\nu$ is smooth, absolute integrability is
automatic. In practice, we rarely have to worry about integrability
when doing differential topology.

The upshot is that {\bf $k$-forms are meant to be integrated on $k$-manifolds}.
Sometimes these are stand-alone abstract $k$-manifolds, sometimes they are
$k$-dimensional submanifolds of larger manifolds, and sometimes they are 
concrete $k$-manifolds embedded in $\R^N$. 

Finally, a technical point. If $X$ is 0-dimensional, then we can't construct
orientation-preserving maps from $\R^0$ to the connected components of $X$. 
Instead, we just take $\int_X \alpha = \sum_{x \in X} \pm \alpha(x)$, where
the sign is the orientation of the point $x$. This follows the general 
principle that reversing the orientation of a manifold should flip the sign
of integrals over that manifold. 

\smallskip

\noindent {\bf Exercise 10:} Let $X= S^1 \subset \R^2$ be the unit circle,
oriented as the boundary of the unit disk. Compute 
$\int_X (x dy - y dx)$ by explicitly pulling this back to $\R$ with an 
orientation-preserving chart and integrating over $\R$. (Which is how you
learned to do line integrals way back in calculus.) [Note: don't worry
about using multiple charts and partitions of unity. Just use a single chart for
the unit circle minus a point.] 

\smallskip

\noindent {\bf Exercise 11:} Now do the same thing one dimension up. 
Let $Y= S^2 \subset \R^3$ be the unit sphere,
oriented as the boundary of the unit ball. Compute 
$\int_X (x dy \wedge dz + y dz \wedge dx + z dx \wedge dy)$ 
by explicitly pulling this back to a subset of $\R^2$ with an 
orientation-preserving chart and integrating over that subset of $\R^2$. 
As with the previous exercise, you can use a single coordinate patch that 
leaves out a set of measure zero, which doesn't contribute to the integral.
Strictly speaking this does {\em not} follow the rules listed above, 
but I'll show you how to clean it up in class. 

\chapter{Stokes' Theorem}

\section{Stokes' Theorem on Euclidean Space}

Let $X=H^n$, the half space in $\R^n$. Specifically, $X = \{ x \in \R^n | 
x_n \ge 0 \}$. Then $\partial X$, viewed as a set, is the standard 
embedding of $\R^{n-1}$ in $\R^n$. However, the orientation on $\partial X$
is {\em not} necessarily the standard orientation on $\R^{n-1}$. Rather, it 
is $(-1)^n$ times the standard orientation on $\R^{n-1}$, since it take $n-1$
flips to change $(n,e_1,e_2,\ldots,e_{n-1})=(-e_n,e_1,\ldots,e_{n-1})$ to 
$(e_1,\ldots,e_{n-1},-e_n)$, which is a negatively oriented basis for $\R^n$. 
(By the way, here's a mnemonic for remembering how to orient boundaries:
ONF = One Never Forgets = Outward Normal First.\footnote{Hat tip to Dan Freed}) 

\begin{thm}[Stokes' Theorem, Version 1] Let $\omega$ be any 
compactly-supported $(n-1)$-form on $X$. Then 
\be \label{Stokes-1} \int_X d \omega = \int_{\partial X} \omega. \ee
\end{thm}

\begin{proof} Let $I_j$ be the ordered subset of $\{1,\ldots,n\}$ in which
the element $j$ is deleted. Suppose that the $(n-1)$-form $\omega$ can 
be expressed as
$\omega(x) = \omega_j(x) dx^{I_j}$, where $\omega_j(x)$ is 
a compactly supported function. Then $d \omega = (-1)^{j-1} \partial_j \omega_j
dx^1 \wedge \cdots \wedge dx^n$. There are two cases to consider:

If $j<n$, then the restriction to $\partial X$ of $\omega$ is zero, since 
$dx^n=0$, so $\int_{\partial X} \omega = 0$. But then
$$\int_{H^n} \partial_j \omega_j
dx^1 \cdots dx^n = \int \left [ \int_{-\infty}^\infty \partial_j \omega_j(x) \, dx_j
\right ] dx^1\cdots dx^{j-1} dx^{j+1} \cdots dx^n$$
Since $\omega_j$ is compactly supported, the inner integral is zero by
the fundamental theorem of calculus. Both sides of (\ref{Stokes-1}) are then
zero, and the theorem holds. 

If $j=n$, then 
\begin{eqnarray}
\int_X d\omega &=& 
 \int_{H^n} (-1)^{n-1} \partial_n \omega_n(x) \, d^nx \cr 
&=& \int_{\R^{n-1}} 
\left [ \int_0^\infty (-1)^{n-1} \partial_n \omega_n(x_1,\ldots,x_n)\, dx^n \right ]
dx^1 \cdots dx^{n-1} \cr 
& = & \int_{\R^{n-1}} (-1)^{n} \omega_n(x_1,\ldots,x_{n-1},0) dx^1\cdots dx^{n-1} \cr
& = & \int_{\partial X} \omega.
\end{eqnarray}
Here we have used the fundamental theorem of calculus and the fact that 
$\omega_n$ is compactly supported to get $\int_0^\infty \partial_n \omega_n(x)dx^n
= - \omega_n(x^1,\ldots,x^{n-1},0)$. 

Of course, not every $(n-1)$-form can be written as $\omega_j dx^{I_j}$ with
$\omega_j$ compactly supported. However, every compactly-supported $(n-1)$-form
can be written as a finite {\em sum} of such terms, one for each
value of $j$. Since equation
(\ref{Stokes-1}) applies to each term in the sum, it also applies to the total. 
\end{proof}

The amazing thing about this proof is how easy it is! The only
analytic ingredients are Fubini's Theorem (which allows us to first
integrate over $x^j$ and then over the other variables) and the
1-dimensional Fundamental Theorem of Calculus. The hard work came earlier,
in developing the appropriate definitions of forms and integrals.

\section{Stokes' Theorem on Manifolds}

Having so far avoided all the geometry and topology of manifolds by working on 
Euclidean space, we now turn back to working on manifolds. Thanks to the 
properties of forms developed in the previous set of notes, everything will
carry over, giving us

\begin{thm}[Stokes' Theorem, Version 2] Let $X$ be a compact oriented
$n$-manifold-with-boundary, and let $\omega$ be an $(n-1)$-form on $X$. Then
\be \label{Stokes-2} \int_X d\omega = \int_{\partial X} \omega, \ee
where $\partial X$ is given the boundary orientation and where the 
right hand side is, strictly 
speaking, the integral of the pullback of $\omega$ to $\partial X$ 
by the inclusion map. 
\end{thm}

\begin{proof} Using a partition-of-unity, we can write $\omega$ as a finite 
sum of forms $\omega_i$, each of which is compactly supported within a single
coordinate patch.  To spell that out, 
\begin{itemize}
\item Every point has a coordinate neighborhood. 
\item Since $X$ is compact, a finite number of such neighborhoods cover $X$. 
\item Pick a partition of unity $\{ \rho_i\}$ subordinate to this cover. 
\item Let $\omega_i = \rho_i \omega$. Since $\sum_i \rho_i=1$,  
$\omega=\sum_i \omega_i$. 
\end{itemize}

Now suppose that the support of $\omega_i$ is contained in the image of 
an orientation-preserving  parametrization $\psi_i: U_i \to X$, 
where $U_i$ is an open set in $H^n$. But then
\begin{eqnarray}
 \int_X d \omega_i &=& \int_{\psi(U_i)} d\omega_i \cr 
&=& 
\int_{U_i} \psi_i^*(d\omega_i) \cr 
& = & \int_{U_i} d(\psi^*\omega_i) \cr 
& = & \int_{H^{n}} d(\psi^*\omega_i) \cr
& = & \int_{\partial H^n} \psi^*\omega_i \cr 
& = & \int_{\partial X} \omega_i,
\end{eqnarray}
where we have used (a) the definition of integration of forms on manifolds,
(b) the fact that $d$ commutes with pullbacks, (c) the fact that 
$\psi^* \omega_i$ and 
$d(\psi^* \omega_i)$ can be extended by zero to all of $H^n$, (d)
Stokes' Theorem on $H^n$, and (e) the definition of integration of forms on
manifolds. Finally, we add everything up.  
\begin{eqnarray}
\int_X d \omega & = & \int_X d \sum_i \omega_i  \cr 
& = & \sum_i \int_X d\omega_i \cr 
& = & \sum_i \int_{\partial X} \omega_i \cr 
& = & \int_{\partial X} \sum_i \omega_i \cr & =&  \int_{\partial X}\omega.
\end{eqnarray}
\end{proof}

Note that $X$ being compact is {\em essential}. 
If $X$ isn't compact, then you can still prove Stokes' Theorem for forms that
are compactly supported, but not for forms in general. For instance, 
if $X=[0,\infty)$ and 
$\omega =1$ (a 0-form), then $\int_X d\omega=0$ but $\int_{\partial X} \omega=-1$.

\smallskip

To relate Stokes' Theorem for forms and manifolds to the classical theorems
of vector calculus, we need a correspondence between line integrals, surface
integrals, and integrals of differential forms. 

\noindent {\bf Exercise 1} If $\gamma$ is an oriented path in $\R^3$
and $\vec v(x)$ is a vector field, show that $\int_\gamma
\omega^1_{\vec v}$ is the line integral $\int \vec v \cdot T ds$,
where $T$ is the unit tangent to the curve and $ds$ is arclength
measure. (Note that this works for arbitrary smooth paths, and not
just for embeddings. It makes perfectly good sense to integrate around 
a figure-8.)

\smallskip

\noindent {\bf Exercise 2} If $S$ is an oriented surface in $\R^3$ and 
$\vec v$ is a vector field, show that $\int_S \omega^2_{\vec v}$ is the 
flux of $\vec v$ through $S$.

\smallskip
 
\noindent{\bf Exercise 3} Suppose that $X$ is a compact connected oriented 
1-manifold-with-boundary 
in $\R^n$. (In other words, a path without self-crossings 
from $a$ to $b$, where $a$ and $b$ might
be the same point.) Show that Stokes' Theorem, applied to $X$, is essentially
the Fundamental Theorem of Calculus. 

\smallskip

\noindent {\bf Exercise 4} Now suppose that $X$ is a bounded domain in $\R^2$. 
Write down Stokes' Theorem in this setting and relate it to the classical
Green's Theorem.

\smallskip

\noindent {\bf Exercise 5} Now suppose that $S$ is an oriented surface in
$\R^3$ with boundary curve $C=\partial S$. Let $\vec v$ be a vector field. 
Apply Stokes Theorem to $\omega^1_{\vec v}$ and to $S$, and express the result
in terms of line integrals and surface integrals. This should give you the 
classical Stokes' Theorem. 

\smallskip

\noindent {\bf Exercise 6} 
On $\R^3$, let $\omega = (x^2+y^2) dx \wedge dy + (x+ye^z) dy
\wedge dz + e^x dx \wedge dz.$ Compute $\int_S \omega$, where $S$ is the 
upper hemisphere of the unit sphere. The answer depends on which orientation
you pick for $S$ of course. Pick one, and compute! [Hint: Find an appropriate 
surface $S'$ so that $S-S'$ is the boundary of a 3-manifold. 
Then use Stokes' Theorem
to relate $\int_S \omega$ to $\int_{S'} \omega$.] 

\smallskip

\noindent {\bf Exercise 7} On $\R^2$ with the origin removed, let 
$\alpha = (x dy - y dx)/(x^2+y^2)$. You previously showed that $d\alpha=0$
(aka ``$\alpha$ is closed''). Show that $\alpha$ is not $d$ of any function
(``$\alpha$ is not exact'')

\smallskip

\noindent {\bf Exercise 8} On $\R^3$ with the origin removed, show that 
$\beta = (x dy \wedge dz - y dx \wedge dz + z dx \wedge dy)/(x^2+y^2+z^2)^{3/2}$
is closed but not exact. 

\smallskip

\noindent {\bf Exercise 9} Let $X$ be a compact oriented $n$-manifold (without
boundary), let $Y$
be a manifold, and let $\omega$ be a closed $n$-form on $Y$. 
Suppose that $f_0$ and $f_1$ are homotopic maps $X \to Y$. Show that $\int_X
f_0^*\omega = \int_X f_1^* \omega$. 

\smallskip

\noindent {\bf Exercise 10} Let $f: S^1 \to \R^2-\{0\}$ be a smooth map
whose winding number around the origin is $k$. Show that 
$\int_{S^1} f^* \alpha = 2\pi k$, where $\alpha$ is the form of Exercise 7.

\chapter{Tensors}

\section{What is a tensor?}

Let $V$ be a finite-dimensional vector 
space.\footnote{Or even an infinite-dimensional vector space, if you apply 
appropriate regularity conditions.}
 It could be $\R^n$, it 
could be the tangent
space to a manifold at a point, or it could just be an abstract vector
space. A {\em $k$-tensor} is a map 
$$T: V \times \cdots \times V \to \R$$
(where there are $k$ factors of $V$) that is linear in each factor.\footnote{Strictly speaking, 
this is what is called a {\em contravariant} tensor. There are also
covariant tensors and tensors of mixed type, all of which play a role in differential geometry. 
But for understanding forms, we only need contravariant tensors.}
That is, for fixed $\vec v_2,\ldots, \vec v_k$, 
$T(\vec v_1,\vec v_2,\ldots, \vec v_{k-1}, \vec v_k)$ is a linear function of $\vec v_1$, and for 
fixed $\vec v_1,\vec v_3,\ldots, \vec v_k$, $T(\vec v_1,\ldots, \vec v_k)$ is a linear function of 
$\vec v_2$, and so on. The space of $k$-tensors on $V$ is denoted $\T^k(V^*)$.

\noindent{\bf Examples:}
\begin{itemize}
\item If $V=\R^n$, then the inner product $P(\vec v,\vec w) = \vec v
  \cdot \vec w$ is a 2-tensor. For fixed $\vec v$ it's linear in $\vec
  w$, and for fixed $\vec w$ it's linear in $\vec v$.
\item If $V=\R^n$, $D(\vec v_1,\ldots,\vec v_n) = \det \begin{pmatrix}
\vec v_1 & \cdots & \vec v_n \end{pmatrix}$ is an $n$-tensor. 
\item If $V=\R^n$, $Three(\vec v)=$ ``the 3rd entry of $\vec v$'' 
is a 1-tensor. 
\item A 0-tensor is just a number. It requires no inputs at all to generate an output. 
\end{itemize}
Note that the definition of tensor says {\em nothing} about how things
behave when you rotate vectors or permute their order. The inner product 
$P$ stays the same when you swap the two vectors, but the determinant 
$D$ changes sign when you swap two vectors. Both are tensors.
For a 1-tensor like $Three$, permuting the order of 
entries doesn't even make sense! 

Let $\{\vec b_1,\ldots, \vec b_n\}$ be a basis for $V$. Every 
vector $\vec v \in V$
can be uniquely expressed as a linear combination:
$$ \vec v = \sum_i v^i \vec b_i,$$
where each $v^i$ is a number. Let $\phi^i(\vec v) = v^i$. The map $\phi^i$
is manifestly linear (taking $\vec b_i$ to 1 and all the other basis vectors to
zero), and so is a 1-tensor. In fact, the $\phi^i$'s form a basis for 
the space of 1-tensors. If $\alpha$ is any 1-tensor, then
\begin{eqnarray}
\alpha(\vec v) & =  \alpha(\sum_i v^i \vec b_i) & \cr 
& =  \sum_i v^i \alpha(\vec b_i) & \hbox{by linearity}\cr 
& =  \sum_i \alpha(\vec b_i) \phi^i(\vec v) & \hbox{since } 
v^i=\phi^i(\vec v) \cr 
& =  ( \sum_i \alpha(\vec b_i) \phi^i) (\vec v)& \hbox{by linearity, so} \cr 
\alpha & =  \sum_i \alpha(\vec b_i) \phi^i.&
\end{eqnarray}

A bit of terminology: The space of 1-tensors is called the {\em dual space}
of $V$ and is often denoted $V^*$. The basis $\{\phi^i\}$ for $V^*$ is 
called the {\em dual basis} of $\{b_j\}$.  Note that 
$$\phi^i(\vec b_j) = \delta^i_j := \begin{cases} 1 & i=j \cr 0 & i \ne j \end{cases},$$
and that there is a duality between vectors and 1-tensors (also called
co-vectors).
\begin{eqnarray}
\vec v = \sum v^i \vec b_i && \hbox{ where } v^i=\phi^i(\vec v) \cr 
\alpha = \sum \alpha_j \phi^j && \hbox{ where } \alpha_j = \alpha(\vec b_j) \cr 
\alpha(\vec v) = \sum \alpha_i v^i.
\end{eqnarray}
It is sometimes convenient to express vectors as columns and
co-vectors as rows. The basis vector $\vec b_i$ is represented by a
column with a 1 in the $i$-th slot and 0's everywhere else, while
$\phi^j$ is represented by a row with a 1 in the $j$th slot and the
rest zeroes. Unfortunately, representing tensors of order greater than
2 visually is difficult, and even 2-tensors aren't properly
described by matrices. To handle 2-tensors or higher, you really
need indices.

If $\alpha$ is a $k$-tensor and $\beta$ is an $\ell$-tensor, then we
can combine them to form a $k+\ell$ tensor that we denote $\alpha
\otimes \beta$ and call the {\em tensor product} of $\alpha$ and
$\beta$:
$$ (\alpha \otimes \beta) (\vec v_1,\ldots, \vec v_{k+\ell}) = \alpha(\vec v_1,\ldots,\vec v_k) \beta(\vec v_{k+1},\ldots, \vec v_{k+\ell}).$$
For instance, 
$$ (\phi^i\otimes \phi^j) (\vec v, \vec w) = \phi^i(\vec v)\, 
\phi^j(\vec w) = v^i w^j.$$
Not only are the $\phi^i \otimes \phi^j$'s 2-tensors, but they form a
basis for the space of 2-tensors. The proof is a generalization of the
description above for 1-tensors, and a specialization of the following
exercise.

\smallskip

\noindent{\bf Exercise 1:} For each ordered k-index
$I=\{i_1,\ldots,i_k\}$ (where each number can range from 1 to $n$),
let $\tilde \phi^I = \phi^{i_1} \otimes \phi^{i_2} \otimes \cdots \otimes
\phi^{i_k}$. Show that the $\tilde \phi^I$'s form a basis for $\T^k(V^*)$,
which thus has dimension $n^k$.  [Hint: If $\alpha$ is a
$k$-tensor, let $\alpha_I=\alpha(\vec b_{i_1},\ldots, \vec b_{i_k})$.
Show that $\alpha = \sum_I \alpha_I \tilde \phi^I$. This implies that the
$\tilde \phi^I$'s span $\T^k(V^*)$.  Use a separate argument to show
that the $\tilde \phi^I$'s are linearly independent.]

\smallskip

Among the $k$-tensors, there are some that have special properties
when their inputs are permuted. For instance, the inner product is
symmetric, with $\vec v \cdot \vec w = \vec w \cdot \vec v$, while the
determinant is anti-symmetric under interchange of any two entries. We
can always decompose a tensor into pieces with distinct symmetry.

For instance, suppose that $\alpha$ is an arbitrary 2-tensor. Define
$$ \alpha_+(\vec v, \vec w) = \frac12 \left ( \alpha(\vec v, \vec w) + \alpha(\vec w, \vec v) \right ); 
\qquad  \alpha_-(\vec v, \vec w) = \frac12 \left ( \alpha(\vec v, \vec w) - \alpha(\vec w, \vec v) \right ).$$
Then $\alpha_+$ is symmetric, $\alpha_-$ is anti-symmetric, and 
$\alpha = \alpha_+ + \alpha_-$. 

\section{Alternating Tensors}

Our goal is to develop the theory of differential forms. But $k$-forms are
made for integrating over $k$-manifolds, and integration means measuring volume.
So the $k$-tensors of interest should behave qualitatively like the determinant
tensor on $\R^k$, which takes $k$ vectors in $\R^k$ and returns the (signed)
volume of the parallelpiped that they span. In particular, it should 
change sign whenever two arguments are interchanged. 

Let $S_k$ denote the group of permutations of $(1,\ldots,k)$. A typical
element will be denoted $\sigma = (\sigma_1, \ldots, \sigma_k)$. The {\em sign}
of $\sigma$ is +1 if $\sigma$ is an even permutation, i.e. the product of 
an even number of transpositions, and $-1$ if $\sigma$ is an odd permutation. 

We say that a $k$-tensor $\alpha$ is {\em alternating} if, for any 
$\sigma \in S_k$ and any (ordered) collection $\{\vec v_1, \ldots, \vec v_k\}$ 
of vectors in $V$, 
$$ \alpha(\vec v_{\sigma_1},\ldots, \vec v_{\sigma_k}) = 
\hbox{sign}(\sigma) \alpha(\vec v_1, \ldots, v_k).$$
The space of alternating $k$-tensors on $V$ is denoted $\Lambda^k(V^*)$.
Note that $\Lambda^1(V^*)=\T^1(V^*)=V^*$ and that $\Lambda^0(V^*)=\T^0(V^*)=\R$.

If $\alpha$ is an arbitrary $k$-tensor, we define
$$ Alt(\alpha) = \frac{1}{k!} \sum_{\sigma \in S_k} \hbox{sign}(\sigma)
\alpha \circ \sigma,$$
or more explicitly
$$ Alt(\alpha)(\vec v_1, \ldots, \vec v_k) = \frac{1}{k!} \sum_{\sigma \in S_k}
\hbox{sign}(\sigma) \alpha(\vec v_{\sigma_1},\ldots, \vec v_{\sigma_k}).$$

\smallskip
\noindent{\bf Exercise 2:} (in three parts) 
\begin{enumerate}
\item Show that $Alt(\alpha) \in \Lambda^k(V^*)$. 
\item Show that $Alt$, restricted to $\Lambda^k(V^*)$, is the identity.
Together with (1), this implies that $Alt$ is a projection from 
$\T^k(V^*)$ to $\Lambda^k(V^*)$.  
\item Suppose that $\alpha$ is a $k$-tensor with $Alt(\alpha)=0$ and that 
$\beta$ is an arbitrary $\ell$-tensor. Show that $Alt(\alpha \otimes \beta)=0$.
\end{enumerate}

\smallskip

Finally, we can define a product operation on alternating tensors. If
$\alpha \in \Lambda^k(V^*)$ and $\beta \in \Lambda^\ell(V^*)$, define
$$ \alpha \wedge \beta = C_{k,\ell} \, Alt(\alpha \otimes \beta),$$
where $C_{k,\ell}$ is an appropriate constant that depends only on $k$ 
and $\ell$.  

\smallskip

\noindent{\bf Exercise 3:} Suppose that $\alpha \in \Lambda^k(V^*)$ and 
$\beta \in \Lambda^\ell(V^*)$, and that 
$C_{k,\ell}=C_{\ell,k}$ Show that $\beta \wedge \alpha = (-1)^{k\ell}
\alpha \wedge \beta$. In other words, wedge products for alternating tensors
have the same symmetry properties as wedge products of forms.  

\smallskip

Unfortunately, there are two different conventions for what the constants
$C_{k,\ell}$ should be! 
\begin{enumerate} 
\item Most authors, including Spivak, use $C_{k,\ell} = 
\frac{(k+\ell)!}{k!\ell!} = { k+\ell \choose k}$. 
The advantage of this convention is that 
$\det = \phi^1 \wedge \cdots \wedge \phi^n$. The disadvantage of this convention
is that you have to keep track of a bunch of factorials when doing wedge 
products.  
\item Some authors, including Guillemin and Pollack, use $C_{k,\ell}=1$. 
This keeps that algebra of wedge products simple, but has the drawback that
$\phi^1\wedge\cdots\wedge \phi^n(\vec b_1,\ldots, \vec b_n) = 1/n!$ instead of 
1. The factorials then reappear in formulas for volume and integration.
\item My personal preference is to use $C_{k,\ell}= {k+\ell \choose k}$, and
that's what I'll do in the rest of these notes. So be careful when 
transcribing formulas from Guillemin and Pollack, since they may differ 
by some factorials! 
\end{enumerate}

\noindent{\bf Exercise 4:} Show that, for both conventions, 
$C_{k,\ell} C_{k+\ell,m}= C_{\ell,m}C_{k,\ell+m}$. 

\smallskip

\noindent{\bf Exercise 5:} Suppose that the constants $C_{k,\ell}$ are chosen
so that $C_{k,\ell} C_{k+\ell,m}= C_{\ell,m}C_{k,\ell+m}$, and suppose that 
$\alpha$, $\beta$ and $\gamma$ are in $\Lambda^k(V^*)$, $\Lambda^\ell(V^*)$ and 
$\Lambda^m(V^*)$, respectively. Show that 
$$(\alpha \wedge \beta) \wedge \gamma = \alpha \wedge (\beta \wedge \gamma).$$
[If you get stuck, look on page 156 of Guillemin and Pollack].

\noindent{\bf Exercise 6:} Using the convention $C_{k,\ell} = 
\frac{(k+\ell)!}{k!\ell!} = { k+\ell \choose k}$, show that 
$\phi^{i_1}\wedge\cdots\wedge \phi^{i_k} = k! Alt(\phi^{i_1}\otimes \cdots
\otimes \phi^{i_k})$. (If we had picked $C_{k,\ell}=1$ as in 
Guillemin and Pollack, we would have gotten the same formula, only without
the factor of $k!$.)
 
\smallskip

Let's take a step back and see what we've done.
\begin{itemize}
\item Starting with a vector space $V$ with basis $\{\vec b_i\}$, we 
created a vector space $V^*=\T^1(V^*)=\Lambda^1(V^*)$ with dual basis 
$\{\phi^j\}$.
\item We defined an associative product $\wedge$ with the property that
$\phi^j \wedge \phi^i = - \phi^i \wedge \phi^j$ and with no other relations. 
\item Since tensor products of the $\phi^j$'s span $\T^k(V^*)$, wedge
products of the $\phi^j$'s must span $\Lambda^k(V^*)$. In other words,
$\Lambda^k(V^*)$ is exactly the space that you get by taking formal 
products of the $\phi^i$'s, subject to the anti-symmetry rule.
\item That's {\em exactly} what we did with the formal symbols $dx^j$ to
create differential forms on $\R^n$. The only difference is that the 
coefficients of differential forms are functions rather than real numbers
(and that we have derivative and pullback operations on forms).
\item Carrying over our old results from wedges of $dx^i$'s, we conclude that 
$\Lambda^k(V^*)$ had dimension ${n \choose k} = \frac{n!}{k!(n-k)!}$
and basis $\phi^I := \phi^{i_1}\wedge \cdots \wedge \phi^{i_k}$, where
$I=\{i_1,\ldots,i_k\}$ is an arbitrary subset of $(1,\ldots,n)$ (with $k$
distinct elements) {\em placed in increasing order}. 
\item Note the difference between $\tilde \phi^I = \phi^{i_1}\otimes
\cdots \otimes \phi^{i_k}$ and $\phi^I = \phi^{i_1}
\wedge \cdots \wedge \phi^{i_k}$. The tensors $\tilde \phi^I$ form a basis
for $\T^k(V^*)$, while the tensors $\phi^I$ form a basis for 
$\Lambda^k(V^*)$. They are related by $\phi^I = k! Alt(\tilde \phi^I)$. 
\end{itemize}

\smallskip

\noindent{\bf Exercise 7:} Let $V=\R^3$ with the standard basis, and let $\pi:
\R^3 \to \R^2$, $\pi(x,y,z)=(x,y)$ be the projection onto the $x$-$y$ plane.
Let $\alpha(\vec v, \vec w)$ be the signed area of the parallelogram spanned
by $\pi(\vec v)$ and  $\pi(\vec w)$ in the $x$-$y$ plane. Similarly,
let $\beta$ and $\gamma$ be
be the signed areas of the projections of $\vec v$ and $\vec w$ in the
$x$-$z$ and $y$-$z$ planes, respectively. Express $\alpha$, $\beta$ and 
$\gamma$ as linear combinations of $\phi^i \wedge \phi^j$'s. [Hint: If you
get stuck, try doing the next two exercises and then come back to this one.]

\smallskip

\noindent {\bf Exercise 8:} Let $V$ be arbitrary.  Show that
$(\phi^{i_1}\wedge \cdots \wedge \phi^{i_k})(\vec b_{j_1},\ldots,\vec
b_{j_k})$ equals $+1$ if $(j_1,\ldots,j_k)$ is an even permutation of
$(i_1,\ldots,i_k)$, $-1$ if it is an odd permutation, and 0 if the two 
lists are not permutations of one another. 

\smallskip 

\noindent{\bf Exercise 9:} Let $\alpha$ be an arbitrary element of 
$\Lambda^k(V^*)$. For each subset $I=(i_1,\ldots,i_k)$ written in increasing
order, let $\alpha_I=\alpha(\vec b_{i_1},\ldots,\vec b_{i_k})$. Show that 
$\alpha = \sum_I \alpha_I \phi^I$. 

\smallskip

\noindent{\bf Exercise 10:} Now let $\alpha_1, \ldots, \alpha_k$ be an
arbitrary ordered list of covectors, and that $\vec v_1, \ldots, \vec v_k$
is an arbitrary ordered list of vectors. Show that 
$(\alpha_1 \wedge \cdots \wedge \alpha_k)(\vec v_1,\ldots,\vec v_k) = 
\det A$, where $A$ is the $k\times k$ matrix whose $i,j$ entry is 
$\alpha_i(\vec v_j)$.  

\section{Pullbacks}

Suppose that $L: V \to W$ is a linear transformation, and that 
$\alpha \in \T^k(W^*)$. We then define the {\em pullback tensor} $L^*\alpha$
by
\be (L^*\alpha)(\vec v_1, \ldots, \vec v_k) = 
\alpha(L(\vec v_1), L(\vec v_2), \ldots, L(\vec v_k)). \ee

This has some important properties. Pick bases $(\vec b_1, \ldots, \vec b_n)$
and $(\vec d_1,\ldots, \vec d_m)$ for $V$ and $W$, respectively, and let 
$\{\phi^j\}$ and $\{\psi^j\}$ be the corresponding dual bases for $V^*$ and 
$W^*$. Let $A$ be the matrix of the linear transformation $L$ relative to
the two bases. That is 
$$ L(\vec v)^j = \sum_j A_{ji} v^i.$$

\noindent{\bf Exercise 11:} Show that the matrix of $L^*: W^* \to V^*$,
relative to the bases $\{\psi^j\}$ and $\{\phi^j\}$, is $A^T$. [Hint:
to figure out the components of a covector, act on a basis vector]

\noindent {\bf Exercise 12:} If $\alpha$ is a $k$-tensor, and if 
$I=\{i_1,\ldots,i_k\}$, show that 
$$ (L^* \alpha)_I = \sum_{j_1,\ldots,j_k} A_{j_1,i_1} A_{j_2,i_2}\cdots
A_{j_k,i_k} \alpha_{(j_1,\ldots,j_k)}.
$$

\noindent {\bf Exercise 13:} Suppose that $\alpha$ is alternating. Show 
that $L^*\alpha$ is alternating. That is, $L^*$ restricted to $\Lambda^k(W^*)$
gives a map to $\Lambda^k(V^*)$. 

\noindent {\bf Exercise 14:} If $\alpha$ and $\beta$ are alternating tensors 
on $W$, show that $L^*(\alpha \wedge \beta) = (L^*\alpha) \wedge (L^* \beta)$.

\section{Cotangent bundles and forms}

We're finally ready to define forms on manifolds. Let $X$ be a $k$-manifold. 
An $\ell$-dimensional {\em vector bundle} over $X$ is a manifold $E$ 
together with a surjection
$\pi: E \to X$ such that 
\begin{enumerate}
\item The preimage $\pi^{-1}(p)$ of any point $p \in X$ is an $n$-dimensional
real vector space. This vector space is called the {\em fiber} over $p$. 
\item For every point $p \in X$ there is a neighborhood $U$ and 
a diffeomorphism $\phi_U: \pi^{-1}(U) \to U \times \R^n$, such that for
each $x \in U$, $\phi_U$ restricted to $\pi^{-1}(x)$ is a linear isomorphism
from $\pi^{-1}(x)$ to $x \times \R^n$ (where we think of $x \times \R^n$ as
the vector space $\R^n$ with an additional label $x$.)
\end{enumerate}
In practice, the isomorphism $\pi^{-1}(U) \to U \times \R^n$ is
usually accomplished by defining a basis $(\vec v_1(x), \ldots, \vec
v_n(x))$ for the fiber over $x$, such that each $\vec v_i$ is a smooth
map from $U$ to $\pi^{-1}(U)$.

Here are some examples of bundles:
\begin{itemize}
\item The tangent bundle $T(X)$. In this case $n=k$. If $\psi$ is a local
parametrization around a point $p$, 
then $d\psi$ applied to $e_1,\ldots,e_n$ give a basis for $T_pX$.
\item The trivial bundle $X \times V$, where $V$ is any $n$-dimensional
vector space. Here we can pick a constant basis for $V$. 
\item The normal bundle of $X$ in $Y$ (where $X$ is a submanifold of $Y$).
\item The {\em cotangent bundle} whose fiber over $x$ is the dual 
space of $T_x(X)$.
This is often denoted $T^*_x(X)$, and the entire bundle is denoted $T^*(X)$.
Given a smoothly varying basis for $T_x(X)$, we can take the dual basis
for $T_x^*(X)$. 
\item The $k$-th tensor power of $T^*(X)$, which we denote $\T^k(T^*(X))$,
i.e. the vector bundle whose fiber
over $x$ is $\T^k(T_x^*(X))$. 
\item The alternating $k$-tensors in $\T^k(T^*(X))$, which we denote
$\Lambda^k(T^*(X))$. 
\end{itemize}

Some key definitions:
\begin{itemize}

\item A {\em section} of a vector bundle $E \to X$ is a smooth map
$s: X \to E$ such that $\pi\circ s$ is the identity on $X$. In other words,
such that $s(x)$ is an element of the fiber over $x$ for every $x$. 

\item A {\em differential form} of degree $k$ is a section of 
$\Lambda^k(T^*(X))$. The (infinite-dimensional) space of $k$-forms on 
$X$ is denoted $\Omega^k(X)$.

\item If $f: X \to \R$ is a function, then $df_x: T_x(X) \to T_{f(x)}(\R)=\R$
is a covector at $X$. Thus every function $f$ defines a 1-form $df$. 

\item If $f: X \to Y$ is a smooth map of manifolds, then $df_x$ is a linear map
$T_x(X)\to T_{f(x)}(Y)$, and so induces a {\em pullback map}
$f^*: \Lambda^k(T_{f(x)}^*(Y)) \to \Lambda^k(T_x^*(X))$, and hence a linear map
(also denoted $f^*$) from $\Omega^k(Y)$ to $\Omega^k(X)$. 
\end{itemize}

\noindent{\bf Exercise 15:} If $f: X \to Y$ and $g: Y \to Z$ are smooth maps
of manifolds, then $g \circ f$ is a smooth map $X \to Z$. Show that 
$(g \circ f)^* = f^* \circ g^*$.  
 
\section{Reconciliation}

We have developed two different sets of definitions for forms, pullbacks,
and the $d$ operator. Our task in this section is to see how they're really
saying the same thing. 

\bigskip

\centerline{Old definitions:}

\smallskip

\begin{itemize}
\item A differential form on $\R^n$ is a formal sum $\sum \alpha_I(x) dx^I$,
where $\alpha_I(x)$ is an ordinary function and $dx^I$ is a product
$dx^{i_1}\wedge \cdots\wedge dx^{i_k}$ of meaningless symbols that
anti-commute.
\item The exterior derivative is 
$d\alpha = \sum_{I,j} (\partial_j \alpha_I(x)) dx^j \wedge dx^I$.
\item If $g: \R^n \to \R^m$, then the pullback operator $g$ is designed
to pull back functions, commute with $d$, and respect wedge products:
If $\alpha = \sum \alpha_I(y) dy^I$, then 
$$ g^*(\alpha)(x) = \sum_I \alpha_I(g(x)) dg^{i_1}\wedge\cdots\wedge dg^{i_k}.
$$
\item Forms on $n$-manifolds are defined via forms on $\R^n$ and local
coordinates and have no intrinsic meaning. 
\end{itemize}

\bigskip

\centerline{New definitions:}

\smallskip

\begin{itemize}
\item A differential form on $\R^n$ is a section of $\Lambda^k(T^*(\R^n))$. 
Its value at each point $x$ is an alternating tensor that takes $k$ 
tangent vectors at that point as inputs and outputs a number. 
\item The exterior derivative on functions is defined as the usual derivative
map $df: T(X) \to \R$. We have not yet defined it for higher-order forms.
\item If $g: X \to Y$, then the pullback map $\Omega^k(Y) \to
  \Omega^k(X)$ is induced by the derivative map $dg: T(X) \to T(Y)$.
\item Forms on manifolds do not require a separate definition from forms
on $\R^n$, since 
tangent spaces, dual spaces, and tensors on tangent spaces are already
well-defined.
\end{itemize}

\bigskip

Our strategy for reconciling these two sets of definitions is:
\begin{enumerate}
\item Show that forms on $\R^n$ are the same in both definitions. 
\item Extend the new definition of $d$ to cover all forms, and show that it
agrees with the old definition on Euclidean spaces.
\item Show that the new definition of pullback, restricted to
  Euclidean spaces, satisfies the same axioms as the old definition,
  and thus gives the same operation on maps between Euclidean spaces,
  and in particular for change-of-coordinate maps.
\item Show that the functional relations that were assumed when we extended
the old definitions to manifolds are already satisfied by the new definitions.
\item Conclude that the new definitions give a concrete realization of the 
old definitions. 
\end{enumerate}

On $\R^n$, the standard basis for the tangent space is $\{\vec e_1,
\ldots, \vec e_n\}$. Since $\partial x^j/\partial x^i = 
\delta_i^j$, $dx^j$ maps
$\vec e_j$ to 1 and maps all other $\vec e_i$'s to zero. Thus the
covectors $d_xx^1, \ldots, d_xx^n$ (meaning the derivatives of the
functions $x^1, \ldots, x^n$ at the point $x$) form a basis for
$T_x^*(\R^n)$ that is dual to $\{\vec e_1,\ldots,\vec e_n\}$. {\bf In
  other words, $\phi^i = dx^i$!!}  The meaningless symbols $dx^i$ of
the old definition are nothing more (or less) than the dual basis of
the new definition. A new-style form is a linear combination $\sum_I
\alpha_I \phi^I$ and an old-style form was a linear combination
$\sum_I \alpha_I dx^I$, so the two definitions are exactly 
the same on $\R^n$. This completes step 1.

Next we want to extend the (new) definition of $d$ to cover arbitrary forms.
We would like it to satisfy $d(\alpha \wedge \beta) = (d\alpha) \wedge \beta
+ (-1)^k \alpha \wedge d\beta$ and $d^2=0$, and that is enough. 
\begin{eqnarray}
d(dx^i)  =  d^2(x^i) & =& 0 \cr 
d(dx^i \wedge dx^j)  = d(dx^i)\wedge dx^j - dx^i\wedge d(dx^j) & =& 0, 
\hbox{and similarly} \cr 
d(dx^I) & = & 0 \hbox{ by induction on the degree of $I$.} \end{eqnarray}
This then forces us to take
\begin{eqnarray}
d (\sum_I \alpha_I dx^I)  &=& \sum_I (d \alpha_I) \wedge dx^I + \alpha_I d(dx^I) \cr & =& \sum_I (d\alpha_I) dx^I \cr 
& =& \sum_{I,j} (\partial_j \alpha_I) dx^j \wedge dx^I,  
\end{eqnarray}
which is exactly the same formula as before. 
Note that this construction also works to define $d$
uniquely on manifolds, as 
long as we can find functions $f^i$ on a neighborhood of a point $p$ 
such that the $df^i$'s span $T^*_p(X)$. But such functions are always available
via the local parametrization. If $\psi:U \to X$ is a local parametrization,
then we can just pick $f^i$ to be the $i$-th entry of $\psi^{-1}$. That is 
$f^i=x^i \circ \psi^{-1}$. This gives a formula for $d$ on $X$ that is equivalent
to ``convert to $\R^n$ using $\psi$, compute $d$ in $\R^n$, and then convert
back'', which was our old definition of $d$ on a manifold. 

We now check that the definitions of pullback are the same. Let 
$g: \R^n \to \R^m$.  
Under the new definition, $g^*(dy^i) (\vec v) = dy^i(dg(\vec v))$, which 
is the $i$th entry of $dg(\vec v)$, where we are using coordinates $\{y^i\}$
on $\R^m$. But that is the same as $dg^i(\vec v)$,
so $g^*(dy^i)=dg^i$. Since the pullback of a function $f$ is just the 
composition $f\circ g$,
and since $g^*(\alpha \wedge \beta)=(g^*(\alpha)) \wedge (g^*(\beta))$ (see
the last exercise in the ``pullbacks'' section), we must have
$$ g^*(\sum_I \alpha_I dy^I)(x) = \sum_I \alpha_I(g(x)) dg^{i_1}\wedge
\cdots \wedge dg^{i_k},$$
exactly as before.  This also shows that $g^*(d\alpha) = d(g^*\alpha)$, since
that identity is a consequence of the formula for $g^*$.

Next we consider forms on manifolds. Let $X$ be an $n$-manifold, let 
$\psi: U \to X$ be a parametrization, where $U$ is an open set in $\R^n$.
Suppose that $a \in U$, and let $p=\psi(a)$. 
The standard bases for $T_a(\R^n)$ and $T_a^*(\R^n)$ are $\{\vec e_1,
\ldots, \vec e_n\}$ and $\{dx^1,\ldots,dx^n\}$. 
Let $\vec b_i = dg_0(\vec e_i)$. The vectors $\{ \vec b_i\}$ form a basis
for $T_p(X)$. Let $\{ \phi^j\}$ be the dual basis. But then
\begin{eqnarray} \psi^*(\phi^j) (\vec e_i) & = & \phi^j (dg_a(\vec e_i))\cr
& = & \phi^j (\vec b_i) \cr 
& = & \delta^j_i \cr 
& = & dx^j(\vec e_i), \hbox{ so } \cr 
\psi^*(\phi^j) & = & dx^j.
\end{eqnarray}
Under the old definition, forms on $X$ were abstract objects that corresponded,
via pullback, to forms on $U$, such that changes of coordinates followed
the rules for pullbacks of maps $\R^n \to \R^n$. 
Under the new definition, $\psi^*$ {\em 
automatically} pulls a basis for $T_p^*(X)$ to a basis for $T_a^*(\R^n)$, 
and this extends to an isomorphism between forms on a neighborhood of $p$
and forms on a neighborhood of $a$. Furthermore, if $\psi_{1,2}$ are two
different parametrizations of the same neighborhood of $p$, and if 
$\psi_1 = \psi_2 \circ g_{12}$ (so that $g_{12}$ maps the $\psi_1$ coordinates
to the $\psi_2$ coordinates), then we automatically have $\psi_1^* = 
g_{12}^* \circ \psi_2^*$, thanks to Exercise 15.

Bottom line: It is perfectly legal to do forms the old way, treating
the $dx$'s as meaningless symbols that follow certain axioms, and treating
forms on manifolds purely via how they appear in various coordinate 
systems. However, sections of bundles of alternating tensors on $T(X)$ give
an intrinsic realization of the exact same algebra. 
The new definitions allow us to talk about what differential forms actually
{\em are}, and to develop a cleaner intuition on how forms behave. 
In particular, they give a very simple explanation of what integration over
manifolds really means. 

\chapter{Integration}

\section{The whole is the sum of the parts}

Before we go about making sense of integrating forms over manifolds, we 
need to understand what integrating functions over $\R^n$ actually means. 
When somebody writes
$$\int_0^3 e^x dx$$
or 
$$ \int_{\R^2} e^{-(x^2+y^2)} dx\, dy$$
or
$$ \int_R f(x) d^nx,$$
what is actually being computed? 

The simplest case is in $\R$. When we write $\int_a^b f(x) dx$, we have a 
quantity with density $f(x)$ spread out over the interval $[a,b]$. We imagine
breaking that interval into small sub-intervals $[x_0,x_1]$, $[x_1,x_2]$, up to
$[x_{N-1},x_N]$, where $a=x_0$ and $b=x_N$. We then have 
\begin{eqnarray}
\int_a^b f(x) dx & = & \hbox{Amount of stuff in $[a,b]$} \cr 
& = & \sum_{k=1}^N \hbox{Amount of stuff in $[x_{k-1},x_k]$} \cr 
& \approx & \sum_{k=1}^N f(x_k^*) \Delta_k x,
\end{eqnarray}
where $\Delta_k x = x_k-x_{k-1}$ is the length of the $k$th interval, and 
$x_k^*$ is an arbitrarily chosen point in the $k$th interval. As long as 
$f$ is continuous and each interval is small, all values of $f(x)$ in the 
$k$th interval are close to $f(x_k^*)$, so $f(x_k^*) \Delta_k x$ is a good
approximation to the amount of stuff in the $k$th interval. As 
$N \to \infty$ and the intervals are chosen smaller and smaller, the errors
go to zero, and we have 
$$\int_a^b f(x) dx = \lim_{N \to \infty} \sum_{k=1}^N f(x_k^*) \Delta_k x.$$
Note that I have not required that all of the intervals $[x_{k-1},x_k]$ be
the same size! While that's convenient, it's not actually necessary. 
All we need for convergence is for all of the sizes to go to zero
in the $N \to \infty$ limit. 

The same idea goes in higher dimensions, when we want to integrate 
any continuous bounded function over any 
bounded region. We break the region into tiny pieces,
estimate the contribution of each piece, and add up the contributions.
As the pieces are chosen smaller and smaller, the errors in our estimates
go to zero, and the limit of our sum is our exact integral. 

If we want to integrate an unbounded function, or integrate over an
unbounded region, we break things up into bounded pieces and add up
the integrals over the (infinitely many) pieces.  
A function is (absolutely) integrable if
the pieces add up to a finite sum, no matter how we slice up the
pieces. Calculus books sometimes distinguish between ``Type I'' improper 
integrals
like $\int_1^\infty x^{-3/2} dx$ and ``Type II'' improper integrals like
$\int_0^1 y^{-1/2} dy$, but they are really the same. Just apply the 
change of variables $y=1/x$:  
\begin{eqnarray}
\int_1^\infty x^{-3/2} dx & = & \sum_{k=1}^\infty \int_k^{k+1} x^{-3/2} dx \cr 
& = & \sum_{k=1}^\infty \int_{1/(k+1)}^{1/k} y^{-1/2} dy \cr 
&=& \int_0^1 y^{-1/2} dy.
\end{eqnarray}

When doing such a change of variables, the width of the intervals can change
drastically. $\Delta y$ is not $\Delta x$, and $x$-intervals of size 1 turn
into $y$-intervals of size $\frac{1}{k(k+1)}$. Likewise, the integrand is not
the same. However, the contribution of the interval, whether written as 
$x^{-3/2} \Delta x$ or $y^{-1/2}\Delta y$, {\em is} the same (at least in
the limit of small intervals). 

In other words, we need to stop thinking about $f(x)$ and $dx$ separately, 
and think instead of the combination $f(x) dx$, which is a machine for 
extracting the contribution of each small interval. 

But that's exactly what the differential form $f(x) dx$ is for! In one
dimension, the covector $dx$ just gives the value of a vector in
$\R^1$. If we evaluate $f(x) dx$ at a sample point $x_k^*$ and apply it to
the vector $x_k-x_{k-1}$, we get
$$ f(x) dx (\vec x_k-\vec x_{k-1}) = f(x_k^*) \Delta_k x.$$

\section{Integrals in 2 or More Dimensions}

Likewise, let's try to interpret the integral of $f(x,y) dx dy$ over
a rectangle $R=[a,b] \times [c,d]$ in $\R^2$. The usual approach is to 
break the interval $[a,b]$ into $N$ pieces and the interval $[c,d]$ into
$M$ pieces, and hence the rectangle $R$ into $NM$ little rectangles with
vertices at $(x_{i-1},y_{j-1})$, $(x_i,y_{j-1})$, $(x_{i-1},y_j)$ and $(x_i,y_j)$,
where $i=1,\ldots,N$ and $j=1,\ldots,M$. 

So what is the contribution of the $(i,j)$-th sub-rectangle $R_{ij}$?
We evaluate $f(x,y)$ at a sample point $(x_i^*,y_j^*)$ and multiply by
the area of $R_{ij}$. However, that area is exactly what you get from
applying $dx \wedge dy$ to the vectors $\vec v_1=(x_i-x_{i-1},0)$ and
$\vec v_2=(0,y_j-y_{j-1})$ that span the sides of the rectangle. In
other words, $f(x_i^*,y_j^*) \Delta_ix \Delta_j y$ is exactly what you
get when you apply the 2-form $f(x,y) dx\wedge dy$ to the vectors
$(\vec v_1,\vec v_2)$ at the point $(x_i^*,y_j^*)$. [Note that this
interpretation requires the normalization
$C_{k,\ell}=\frac{(k+\ell)!}{k!\ell!}$ for wedge products. If we had
used $C_{k,\ell}=1$, as in Guillemin and Pollack, then $dx \wedge
dy(\vec v_1, \vec v_2)$ would only be half the area of the rectangle.]

The same process works for integrals over any bounded domain $R$ in $\R^n$.
To compute $\int_R f(x) d^nx$:
\begin{enumerate}
\item Break $R$ into a large number of small pieces $\{ R_I\}$, which we'll
call ``boxes'', 
each of which is 
approximately a parallelpiped spanned by vectors $\vec v_1,\ldots, \vec v_n$,
where the vectors don't have to be the same for different pieces.
\item To get the contribution of a box $R_I$, pick a point $x_I^* \in R_I$,
evaluate the $n$-form $f(x) dx^1\wedge\cdots\wedge dx^n$ at $x_I^*$, and
apply it to the vectors $\vec v_1, \ldots, \vec v_k$. Your answer will
depend on the choice of $x_I^*$, but all choices will give approximately the
same answer. 
\item Add up the contributions of all of the different boxes. 
\item Take a limit as the sizes of the boxes go to zero uniformly. Integrability
means that this limit does not depend on the choices of the sample points 
$x_I^*$, or on the way that we defined the boxes. When $f$ is continuous and
bounded, this always works. When $f$ is unbounded or discontinuous, or when
$R$ is unbounded, work is required to show that the limit is well-defined.
\end{enumerate}

For instance, to integrate $e^{-(x^2+y^2)}dx dy$ over the unit disk, we need
to break the disk into pieces. One way is to use Cartesian coordinates, where
the boxes are rectangles aligned with the coordinate axes and of size 
$\Delta x \times \Delta y$. Another way is to use polar coordinates, where 
the boxes have $r$ and $\theta$ ranging over small intervals. 

\noindent {\bf Exercise 1:} Let $R_I$ be
a ``polar rectangle'' whose vertices $p_1$, $p_2$, $p_3$ and $p_4$ 
have polar coordinates $(r_0,\theta_0)$,
$(r_0+\Delta r, \theta_0)$, $(r_0, \theta_0+\Delta \theta)$ and 
$r_0+\Delta r, \theta_0+\Delta \theta)$, respectively, 
where we assume that $\Delta r$ is
much smaller than $r_0$ and that $\Delta \theta$ is small in absolute terms.
Let $\vec v_1$ be the vector from $p_1$ to $p_2$ and $\vec v_2$ is the vector
from $p_1$ to $p_3$. \newline
(a) Compute $dx \wedge dy (\vec v_1, \vec v_2)$. \newline
(b) If our sample point $x_I^*$ has polar coordinates $(r^*, \theta^*)$, 
evaluate the approximate contribution of this box. \newline
(c) Express the limit of the sum over all boxes as a double integral over
$r$ and $\theta$. \newline
(d) Evaluate this integral. 

\section{Integration Over Manifolds}

Now let $X$ be an oriented $n$-manifold (say, embedded in $\R^N$), 
and let $\alpha$ be an $n$-form. 
The integral $\int_X \alpha$ is the result of the following process. 

\begin{enumerate}
\item Break $X$ into a number of boxes $X_I$, 
where each box can be approximated as a parallelpiped containing a point 
$p_I^*$,
with the {\em oriented} collection of vectors 
$\vec v_1, \ldots, \vec v_n$ representing the edges. 
\item Evaluate $\alpha$ at $p_I^*$ and apply it to the vectors 
$\vec v_1, \ldots, \vec v_n$. 
\item Add up the contributions of all the boxes. 
\item Take a limit as the size of the boxes goes to zero uniformly. 
\end{enumerate}

In practice, Step 1 
is usually done via a parametrization $\psi$, and letting the box $X_I$ be
the image under $\psi$ of an actual $\Delta x_1 \times \cdots \times
\Delta x_n$ rectangle in $\R^n$, and setting
$\vec v_i=d\psi_a(\Delta x_i \vec e_i)$, where $p_I^*=\psi(a)$. Note that 
$p_I^*$ is not necessarily a vertex. It's just an arbitrary point in the
box. 

If the box is constructed in this way, then 
Step 2 is {\em exactly} the same as applying $\psi^*\alpha(a)$ to the vectors 
$\{\Delta x_i \vec e_i\}$. But that makes integrating $\alpha$ over $X$ 
the same as integrating 
$\psi^*\alpha$ over $\R^n$! This shows directly that different choices of 
coordinates give the same integrals, as long as the coordinate patches are 
oriented correctly. 

When a manifold consists of more than one coordinate patch, there are several
things we can do. One is to break $X$ into several large pieces, each within
a coordinate patch, and then break each large piece into small 
coordinate-based boxes, exactly as described above. 
Another is to use a partition of unity to write 
$\alpha = \sum \rho_i \alpha$ as a sum of pieces supported in a 
single coordinate chart, and then integrate each $\alpha_i$ separately. 

This allows for a number of natural constructions where forms are defined
intrinsically rather than via coordinates.

Let $X$ be an oriented $(n-1)$-manifold in $\R^n$, and let $\vec n(x)$
be the unit normal to $X$ at $x$ whose sign is chosen such that, for any
oriented basis $\vec v_1,\ldots,\vec v_{n-1}$ of $T_xX$, the basis  
$(\vec n, \vec v_1, \ldots, \vec v_{n-1})$ of $T_x\R^n$ is positively oriented. 
(E.g, if $X = \partial Y$, then $n$ is the normal pointing out from $Y$). 
Let $dV = dx^1\wedge \cdots \wedge dx^n$ be the volume form on $\R^n$. Define
a form $\omega$ on $X$ by 
$$\omega(\vec v_1,\ldots,\vec v_{n-1}) = 
dV(\vec n, \vec v_1,\ldots,\vec v_{n-1}).$$

\smallskip

\noindent{\bf Exercise 2:} Show that $\int_X \omega$ is the $(n-1)$-dimensional
volume of $X$. 

\smallskip

More generally, let $\alpha$ be any $k$-form on a manifold $X$, and let
$\vec w(x)$ be any vector field. We define a new $(k-1)$-form $i_w\alpha$
by 
$$ (i_w\alpha)(\vec v_1,\ldots,\vec v_{k-1})=\alpha(\vec w, \vec v_1,
\ldots,\vec v_{k-1}).$$

\smallskip

\noindent{\bf Exercise 3:} Let $S$ be a surface in $\R^3$ and let $\vec v(x)$ be
a vector field. Show {\em directly} that $\int_S i_v(dx\wedge dy\wedge dz)$
is the flux of $\vec v$ through $S$. That is, show that 
$i_v(dx \wedge dy \wedge dz)$ applied to a pair of (small) vectors gives
(approximately) the flux of $\vec v$ through a parallelogram spanned by those
vectors. 

\smallskip

\noindent{\bf Exercise 4:} In $\R^3$ we have already seen $i_v(dx \wedge dy
\wedge dz)$. What did we call it? 

\smallskip

\noindent{\bf Exercise 5:} Let $\vec v$ be any vector field in $\R^n$.
Compute $d(i_v(dx^1\wedge \cdots \wedge dx^n))$. 

\smallskip

\noindent{\bf Exercise 6:} Let $\alpha = \sum \alpha_I(x) dx^I$ be a 
$k$-form on $\R^n$ and let 
$\vec v(x)=\vec e_i$, the $i$-th standard basis vector for $\R^n$. 
Compute $d(i_v \alpha) + i_v(d\alpha)$. Generalize to the case where $\vec v$
is an arbitrary {\em constant} vector field. 

\smallskip

When $\vec v$ is not constant,
the expression $d(i_v \alpha) + i_v(d\alpha)$ is more complicated, and 
depends both on derivatives of $v$ and derivatives of $\alpha_I$, as we 
saw in the last two exercises. This quantity 
is called the {\em Lie derivative} of $\alpha$ with respect to $\vec v$.

It is certainly possible to feed more than one vector field to a 
$k$-form, thereby reducing its degree by more than 1. It immediately
follows that $i_vi_w = - i_wi_v$ as a map 
$\Omega^k(X) \to \Omega^{k-2}(X)$. 

\chapter{de Rham Cohomology}

\section{Closed and exact forms}

Let $X$ be a $n$-manifold (not necessarily oriented), and let $\alpha$ be a
$k$-form on $X$. We say that $\alpha$ is {\em closed} if $d\alpha=0$ and 
say that $\alpha$ is {\em exact} if $\alpha = d\beta$ for some 
$(k-1)$-form $\beta$. (When $k=0$, the 0 form is also considered exact.) Note that
\begin{itemize}
\item Every exact form is closed, since $d(d\beta)=d^2\beta=0$. 
\item A 0-form is closed if and only if it is locally constant, i.e.
constant on each connected component of $X$. 
\item Every $n$-form is closed, since then $d\alpha$ would be an $(n+1)$-form
on an $n$-dimensional manifold, and there are no nonzero $(n+1)$-forms. 
\end{itemize}

Since the exact $k$-forms are a subspace of the closed $k$-forms, we can
defined the quotient space 
$$ H^k_{dR}(X) = \frac{\hbox{Closed $k$-forms on $X$}}{\hbox{Exact $k$-forms
on $X$}}.$$
This quotient space is called the $k$th {\em de Rham cohomology of $X$}. 
Since this is the only kind of cohomology we're going to discuss in these
notes, I'll henceforth omit the prefix ``de Rham'' and the subscript ${}_{dR}$.  If $\alpha$ is a closed
form, we write $[\alpha]$ to denote the class of $\alpha$ in $H^k$, and say
that the {\em form} $\alpha$ represents the {\em cohomology class} $[\alpha]$. 

The wedge product of forms extends to a product operation 
$H^k(X) \times H^\ell(X) \to H^{k+\ell}(X)$. If $\alpha$ and $\beta$ are closed,
then
\begin{eqnarray*}
d(\alpha \wedge \beta) & = & (d\alpha)\wedge \beta + (-1)^k \alpha \wedge d\beta
\cr &=& 0 \wedge \beta \pm \alpha \wedge 0 = 0,
\end{eqnarray*}
so $\alpha \wedge \beta$ is closed. Thus $\alpha \wedge \beta$ represents a 
class in $H^{k+\ell}$, and we define 
$$ [\alpha] \wedge [\beta] = [\alpha \wedge \beta].$$
We must check that this is well-defined. I'm going to spell this out in gory
detail as an example of computations to come.

Suppose that $[\alpha']=[\alpha]$ and $[\beta']=[\beta]$. We must show that
$[\alpha' \wedge \beta']=[\alpha \wedge \beta]$.  However $[\alpha']=[\alpha]$
means that $\alpha'$ and $\alpha$ differ by an exact form, and similarly
for $\beta'$ and $\beta$:
\begin{eqnarray*}
\alpha' &=& \alpha + d\mu \cr 
\beta' &=& \beta + d\nu
\end{eqnarray*}
But then 
\begin{eqnarray*}
\alpha' \wedge \beta' & = & (\alpha + d\mu) \wedge (\beta + d\nu) \cr 
& = & \alpha \wedge \beta + (d\mu)\wedge \beta + \alpha \wedge d\nu
+ d\mu \wedge d\nu \cr 
&=& \alpha \wedge \beta + d(\mu \wedge \beta) + (-1)^k d(\alpha \wedge \nu)
+ d(\mu \wedge d\nu) \cr 
& = & \alpha \wedge \beta + \hbox{exact forms},
\end{eqnarray*}
where we have used the fact that $d(\mu\wedge \beta)=d\mu \wedge \beta
+ (-1)^{k-1}\mu \wedge d\beta = d\mu \wedge \beta$, and similar expansions 
for the other terms. That is, 
\begin{eqnarray*}
\hbox{(Exact)} \wedge \hbox{(Closed)} & = & \hbox{(Exact)} \cr 
\hbox{(Closed)} \wedge \hbox{(Exact)} & = & \hbox{(Exact)} \cr 
\hbox{(Exact)} \wedge \hbox{(Exact)} & = & \hbox{(Exact)}
\end{eqnarray*}
Thus $\alpha'\wedge \beta'$ and $\alpha\wedge \beta$ represent the same 
class in cohomology.    
Since $\beta \wedge \alpha = (-1)^{k\ell} \alpha \wedge \beta$, it also follows 
immediately that $[\beta]\wedge [\alpha] = (-1)^{k\ell} [\alpha]\wedge[\beta]$.
 
We close this section with a few examples. 
\begin{itemize}
\item If $X$ is a point, then $H^0(X)=\Omega^0(X)=\R$, and $H^k(X)=0$ for
all $k\ne 0$, since there are no nonzero forms in dimension greater than 1.
\item If $X=\R$, then $H^0(X)=\R$, since the closed 0-forms are the
constant functions, of which only the 0 function is exact. All 1-forms are
both closed and exact. If $\alpha = \alpha(x) dx$ is a 1-form, then 
$\alpha = df$, where $f(x) = \int_0^x \alpha(s) ds$ is the indefinite integral
of $\alpha(x)$.
\item If $X$ is {\em any} connected manifold, then $H^0(X)=\R$. 
\item If $X = S^1$ (say, embedded in $\R^2$), then $H^1(X)=\R$, and the 
isomorphism is obtained by integration: $[\alpha] \to \int_{S^1} \alpha$. 
If the form $\alpha$ is exact, then $\int_{S^1} \alpha = 0$. Conversely, if 
$\int_{S^1} \alpha=0$, then $f(x)=\int_a^x \alpha$ (for an arbitrary 
fixed starting point $a$) is well-defined and $\alpha=df$. 
\end{itemize}

\section{Pullbacks in Cohomology}

Suppose that $f: X \to Y$ and that $\alpha$ is a closed form on $Y$, 
representing a class in $H^k(Y)$. Then $f^*\alpha$ is also closed, since 
$$ d(f^*\alpha) = f^*(d\alpha) = f^*(0)=0,$$
so $f^*\alpha$ represents a class in $H^k(X)$. If $\alpha'$ also 
represents $[\alpha] \in H^k(Y)$, then we must have $\alpha'=\alpha+d\mu$, so
$$ f^*(\alpha') = f^*\alpha + f^*(d\mu) = f^*\alpha + d(f^*\mu)$$
represents the same class in $H^k(X)$ as $f^*\alpha$ does. We can therefore
define a map 
$$f^\sharp: H^k(Y) \to H^k(X), \qquad f^\sharp[\alpha] = [f^*\alpha].$$
We are using notation to distinguish between the pullback map $f^*$ on 
{\em forms} and the pullback map $f^\sharp$ on {\em cohomology}. Guillemin
and Pollack also follow this convention. However, most authors use 
$f^*$ to denote both maps, hoping that it is clear from context whether 
we are talking about forms or about the classes they represent.  (Still others
use $f^\sharp$ for the map on forms and $f^*$ for the map on cohomology.
Go figure.) 

Note that $f^\sharp$ is a {\em contravariant functor}, which is a fancy way
of saying that it reverses the direction of arrows. If $f: X \to Y$, then
$f^\sharp: H^k(X) \leftarrow H^k(Y)$. If $f: X \to Y$ and $g: Y \to Z$, then
$g^\sharp: H^k(Z) \to H^k(Y)$ and $f^\sharp: H^k(Y) \to H^k(X)$. Since 
$(g \circ f)^* = f^* \circ g^*$, it follows that 
$(g \circ f)^\sharp = f^\sharp \circ g^\sharp$. 

We will have more to say about pullbacks in cohomology after we have established
some more machinery.

\section{Integration over a fiber and the 
Poincare Lemma}

\begin{thm}[Integration over a fiber] Let $X$ be any manifold.
Let the {\em zero section} $s_0:X \to \R \times X$ be given by $s_0(x)=(0,x)$,
and let the {\em projection} $\pi: \R \times X \to X$ be given 
by $\pi(t,x)=x$. Then
$s_0^\sharp: H^k(\R \times X) \to H^k(X)$ and 
$\pi^\sharp: H^k(X) \to H^k(\R \times X)$ are isomorphisms
and are inverses of each other. 
\end{thm}

\begin{proof}
Since $\pi \circ s_0$ is the identity on $X$, $s_0^\sharp \circ \pi^\sharp$ is the
identity on $H^k(X)$. We must show that $\pi^\sharp \circ s_0^\sharp$ is the 
identity on $H^k(\R \times X)$. We do this by constructing 
a map 
$P: \Omega^k(\R \times X) \to \Omega^{k-1}(\R \times X),$ 
called a {\em homotopy operator,} such that for any $k$ form
$\alpha$ on $\R \times X$, 
$$ (1 - \pi^*\circ s^*) = d (P(\alpha)) + P(d\alpha).$$
If $\alpha$ is closed, this implies 
that $\alpha$ and $\pi^*(s_0^*\alpha)$ differ by the exact form $d(P(\alpha))$
and so represent the same class in cohomology,
and hence that $\pi^\sharp \circ s_0^\sharp[\alpha]=[\alpha]$. Since this is 
true for all $\alpha$, $\pi^\sharp \circ s_0^\sharp$ is the identity. 

Every $k$-form on $Y$ can be uniquely written as a product
$$\alpha(t,x) = dt \wedge \beta(t,x) + \gamma(t,x),$$
where $\beta$ and $\gamma$ have no $dt$ factors. The $(k-1)$-form $\beta$ 
can be written as a sum:
$$\beta(t,x) = \sum_J \beta_J(t,x) dx^J,$$
where $\beta_J(t,x)$ is an ordinary function, and 
we likewise write
$$\gamma(t,x) = \sum_I \gamma_I(t,x) dx^I.$$ 
We define
$$P(\alpha)(t,x) = \sum_J (\int_0^t \beta_J(s,x) ds) dx^J.$$
$P(\alpha)$ is called the {\em integral along the fiber} of $\alpha$. 
Note that $s_0^*\alpha$, evaluated at $x$, is $\sum_I \gamma(0,x) dx^I$,
and that 
$$(1-\pi^*s_0^*)\alpha (t,x) = dt \wedge \beta(t,x) + 
\sum_I(\gamma(t,x)-\gamma(t,0)) dx^I.$$ 

Now
we compute $dP(\alpha)$ and $P(d\alpha)$. 
Since
$$ d\alpha(t,x) = -dt \wedge \sum_{j,J} (\partial_j \beta_J(t,x))dx^j\wedge dx^J
+ \sum_{I} \partial_t\gamma_I(t,x) dt\wedge dx^I + \sum_{I,j}
\partial_j \gamma_I(t,x) dx^j \wedge dx^I,$$
where $j$ runs over the coordinates of $X$, we have 
$$P(d\alpha)(t,x) = -\sum_{j,J} \int_0^t \Big (\partial_j \beta_J(s,x)ds \Big )dx^j\wedge
dx^J + \sum_I \Big (\gamma_I(t,x)-\gamma_I(0,x) \Big ) dx^I,$$ 
where we have used $\int_0^t \partial_s \gamma_I(s,x) ds = \gamma_I(t,x)-\gamma_I(0,x)$. 
Meanwhile,
$$d(P(\alpha)) = \sum_{j,J} \left (\int_0^t \partial_j \beta_J(s,x) ds \right ) dx^j
\wedge dx^J + \sum_J \beta_J(t,x) dt \wedge dx^J,$$
so 
$$ (dP+Pd)\alpha(t,x) = \sum_I \Big (\gamma_I(t,x)-\gamma_I(0,x) \Big ) dx^I + 
dt \wedge \beta(t,x) = (1-\pi^*s_0^*) \alpha(t,x).$$
\end{proof}

\noindent{\bf Exercise 1:} In this proof, the operator 
$P$ was defined relative to local coordinates on $X$. Show that this 
is in fact well-defined. That is, if we have two parametrizations
$\psi$ and $\phi$, and we compute $P(\alpha)$ using the $\phi$ coordinates
and then convert to the $\psi$ coordinates, we get the same result as if 
we computed $P(\alpha)$ directly using the $\psi$ coordinates.

\smallskip
 
An immediate corollary of this theorem 
is that $H^k(\R^n)=H^k(\R^{n-1})=\cdots=H^k(\R^0)$. 
In particular,
\begin{thm}[Poincare Lemma] On $\R^n$, or on any manifold diffeomorphic to 
$\R^n$, every closed form of degree 1 or higher is exact. 
\end{thm}

\noindent {\bf Exercise 2:} Show that a vector field $\vec v$ on $\R^3$ is
the gradient of a function if and only if $\nabla \times \vec v=0$ everywhere.

\noindent {\bf Exercise 3:} Show that a vector field $\vec v$ on $\R^3$ 
can be written
as a curl (i.e., $\vec v = \nabla \times \vec w$) if and only if 
$\nabla \cdot \vec v=0$. 

\noindent {\bf Exercise 4:} Now consider the 3-dimensional torus 
$X = \R^3/\Z^3$. Construct a vector
field $\vec v(x)$ 
whose curl is zero that is not a gradient (where we use the local
isomorphism with $\R^3$ to define the curl and gradient). 
Construct a vector field
$\vec w(x)$ whose divergence is zero that is not a curl. 

In the integration-along-a-fiber theorem, we showed that $s_0^\sharp$ was 
the inverse of $\pi^\sharp$. However, we could have used the 1-section
$s_1(x)=(1,x)$ instead of the 0-section and obtained the same result. (Just
replace 0 with 1 everwhere that refers to a value of $t$). Thus
$$s_1^\sharp = (\pi^\sharp)^{-1}=s_0^\sharp.$$
This has important consequences for homotopies.

\begin{thm} Homotopic maps induce the {\em same} map in cohomology. 
That is, if $X$ and $Y$ are manifolds and $f_{0,1}: X \to Y$ are 
smooth homotopic maps, then $f_1^\sharp=f_0^\sharp$.
\end{thm}

\begin{proof}
If $f_0$ and $f_1$ are homotopic, then we can find a smooth map 
$F: \R \times X
\to Y$ such that $F(t,x)=f_0(x)$ for $t \le 0$ and $F(t,x)=f_1(x)$ for 
$t \ge 1$. But then $f_1 = F \circ s_1$ and $f_0=F\circ s_0$. Thus 
$$ f_1^\sharp = s_1^\sharp \circ F^\sharp = s_0^\sharp \circ F^\sharp
= (F \circ s_0)^\sharp = f_0^\sharp.$$
\end{proof}

\noindent{\bf Exercise 5:} Recall that if $A$ is a submanifold of $X$, then a {\em retraction} 
$r: X \to A$ (sometimes just called a {\em retract}) is a smooth map such that $r(a)=a$ for all $a \in A$. If such a map exists, we say that 
$A$ is a {\em retract} of $X$.  Suppose that $r: X \to A$ is such a retraction,
and that $i_A$ be the inclusion
of $A$ in $X$.  Show that $r^\sharp:H^k(A)\to H^k(X)$ is surjective and 
$i_A^\sharp: H^k(X) \to H^k(A)$ is injective
in every degree $k$. [We will soon see that $H^k(S^k)=\R$. This exercise,
combined with the Poincare Lemma, will then
provide another proof that there are no retractions from the unit ball in 
$\R^n$ to the unit sphere.]

\smallskip

\noindent{\bf Exercise 6:} Recall that a {\em deformation retraction} is 
a retraction $r: X \to A$ such that $i_A \circ r$ is homotopic 
to the identity on $X$, in which case we say that $A$ is a {\em deformation retract} of $X$. 
Suppose that $A$ is a deformation retract of 
$X$. Show that $H^k(X)$ and $H^k(A)$ are isomorphic. [This provides another
proof of the Poincare Lemma, insofar as $\R^n$ deformation retracts to a point.]

\section{Mayer-Vietoris Sequences 1: Statement}

Suppose that a manifold $X$ can be written as the union of two open 
submanifolds, $U$ and $V$. 
The Mayer-Vietoris Sequence is a technique for computing the cohomology
of $X$ from the cohomologies of $U$, $V$ and $U \cap V$. This has direct
practical importance, in that it allows us to compute things like 
$H^k(S^n)$ and many other simple examples. It also allows us to prove many
properties of compact manifolds by induction on the number of open sets in 
a ``good cover'' (defined below). Among the things that can be proved with
this technique (of which we will only prove a subset) are:
\begin{enumerate}
\item $H^k(S^n)=\R$ if $k=0$ or $k=n$ and is trivial otherwise. 
\item If $X$ is compact, then $H^k(X)$ is finite-dimensional. This is hardly
obvious, since $H^k(X)$ is the quotient of the infinite-dimensional vector
space of closed $k$-forms by another infinite-dimensional space of exact 
$k$-forms. But as long as $X$ is compact, the quotient is finite-dimensional.
\item If $X$ is a compact, oriented $n$-manifold, then $H^n(X)=\R$. 
\item If $X$ is a compact, oriented $n$-manifold, then $H^k(X)$ is isomorphic
to $H^{n-k}(X)$. (More precisely to the dual of $H^{n-k}(X)$, but every
finite-dimensional vector space is isomorphic to its own dual.) This is 
called Poincare duality. 
\item If $X$ is any compact manifold, orientable or not, then $H^k(X)$ is 
isomorphic to $Hom(H_k(X),\R)$, where $H_k(X)$ is the $k$-th homology group
of $X$. 
\item A formula for $H^k(X \times Y)$ in terms of the cohomologies of $X$ and 
$Y$. 
\end{enumerate}

Suppose we have a sequence 
$$ V^1 \xrightarrow {L_1} V^2 \xrightarrow {L_2} V^3 \xrightarrow {L_3} \cdots$$
where each $V^i$ is a vector space and each $L_i: V^i \to V^{i+1}$ is a linear
transformation. We say that this sequence is {\em exact} if the kernel of 
each $L_i$ equals the image of the previous $L_{i-1}$. In particular, 
$$ 0 \xrightarrow {} V \xrightarrow L W \xrightarrow {} 0 $$
is exact if and only if $L$ is an isomorphism, since the kernel of $L$
has to equal the image of 0, and the image of $L$ has to equal the kernel
of the 0 map on $W$. 

\smallskip

\noindent{\bf Exercise 7:} A {\em short exact sequence} involves three spaces 
and two maps:
$$ 0 \rightarrow U \xrightarrow i V \xrightarrow j W \rightarrow 0 $$
Show that if this sequence is exact, there must be an isomorphism $h:V 
\to U \oplus W$, with $h\circ i(u)=(u,0)$
and $j\circ h^{-1}(u,w)=w$. 

\smallskip

Exact sequences can be defined for homeomorphisms between arbitrary 
Abelian groups, and not just vector spaces, but are much simpler when 
applied to vector spaces. In particular, the analogue of the previous exercise
is {\em false} for groups. (E.g. one can define a short exact sequence 
$0 \to \Z_2 \to \Z_4 \to \Z_2 \to 0$ even though $\Z_4$ is not isomorphic to 
$\Z_2 \times \Z_2$.)

Suppose that $X = U \cup V$, where $U$ and $V$ are open submanifolds of $X$.
There are natural inclusion maps $i_U$ and $i_V$ of $U$ and $V$ into $X$, 
and these induce maps $i_U^*$ and $i_V^*$ from $\Omega^k(X)$ to 
$\Omega^k(U)$ and $\Omega^k(V)$. Note that $i_U^*(\alpha)$ is 
just the restriction of $\alpha$ to $U$, while $i_V^*(\alpha)$ is the 
restriction of $\alpha$ to $V$.
Likewise, there are inclusions
$\rho_U$ and $\rho_V$ of $U \cap V$ in $U$ and $V$, respectively, and 
associated restrictions $\rho_U^*$ and $\rho_V^*$ from $\Omega^k(U)$ and 
$\Omega^k(V)$ to $\Omega^k(U\cap V)$. Together, these form a 
sequence:
\be \label{short-exact-forms} 
0 \rightarrow \Omega^k(X) \xrightarrow {i_k} 
\Omega^k(U) \oplus \Omega^k(V) \xrightarrow {j_k}
\Omega^k(U\cap V) \rightarrow 0,  
\ee
where the maps are defined as follows. If $\alpha \in \Omega^k(X)$,
$\beta \in \Omega^k(U)$ and $\gamma \in \Omega^k(V)$, 
then
\begin{eqnarray*} i_k(\alpha) &=& (i_U^* \alpha, i_V^*\alpha) \cr 
j_k(\beta,\gamma) & = & r_U^*\beta - r_V^*\gamma.
\end{eqnarray*}
Note that $d(i_k(\alpha))=i_{k+1}(d\alpha)$ and that $d(j_k(\beta,\gamma))=
j_{k+1}(d\beta,d\gamma)$. That is, the diagram 
$$   \begin{CD}
 0 @>>>   \Omega^k(X)   @>{i_k}>> \Omega^k(U)\oplus\Omega^k(V) @>{j_k}>> \Omega^k(U \cup V) @>>> 0 \\
 @.    @VV{d_k}V           @VV{d_k}V  @VV{d_k}V  @.\\
0 @>>>  \Omega^{k+1}(X)   @>{i_{k+1}}>> \Omega^{k+1}(U)\oplus\Omega^{k+1}(V) @>{j_{k+1}}>> \Omega^{k+1}(U \cup V)  @>>> 0
   \end{CD}
$$   
commutes. Thus $i_k$ and $j_k$ send closed forms to closed forms
and exact forms to exact forms, and induce maps 
$$ i_k^\sharp: H^k(X) \to H^k(U)\oplus H^k(V); \qquad 
j_k^\sharp: H^k(U) \oplus H^k(V) \to H^k(U\cap V).$$

\begin{thm}[Mayer-Vietoris] There exists a map
$d_k^\sharp: H^k(U \cap V) \to H^{k+1}(X)$ such that the sequence 
$$ \cdots H^k(X) \xrightarrow {i_k^\sharp} H^k(U)\oplus H^k(V) \xrightarrow {j_k^\sharp} H^k(U\cap V) \xrightarrow {d_k^\sharp}  H^{k+1}(X) \xrightarrow {i_{k+1}^\sharp} H^{k+1}(U)\oplus H^{k+1}(V) \rightarrow \cdots $$
is exact. 
\end{thm}

The proof is a long slog, and warrants a section of its own. Then we will
develop the uses of Mayer-Vietoris sequences. 

\section{Proof of Mayer-Vietoris}

The proof has several big steps.
\begin{enumerate}
\item We show that the sequence 
(\ref{short-exact-forms}) of forms is actually exact. 
\item Using that 
exactness, and the fact that $i$ and $j$ commute with $d$, we then construct
the map $d_k^\sharp$. 
\item Having constructed the maps, we show exactness at $H^k(U)\oplus H^k(V)$, i.e., 
that the image of $i_k^\sharp$
equals the kernel of $j_k^\sharp$.
\item We show exactness at $H^k(U\cap V)$, i.e., that the image
of $j_k^\sharp$ equals the kernel of $d_k^\sharp$. 
\item We show exactness at 
$H^{k+1}(X)$, i.e. that the kernel of $i_{k+1}^\sharp$ equals the image of 
$d_{k}^\sharp$.
\item Every step but the first is formal, and applies just as well to 
{\em any} short exact sequence of (co)chain complexes. This construction 
in homological algebra is called the {\em snake lemma}, and may be familiar
to some of you from algebraic topology. If so, you can skip ahead after 
step 2. If not, don't worry. We'll cover everything from scratch. 
\end{enumerate}

{\bf Step 1:}
Showing that $$ 0 \rightarrow \Omega^k(X) \xrightarrow {i_k}
\Omega^k(U) \oplus \Omega^k(V) \xrightarrow {j_k} \Omega^k(U \cap V) \rightarrow
0$$
amounts to showing that $i_k$ is injective, 
that $Im(i_k)=Ker(j_k)$, and that $j_k$ is surjective. 
The first two are easy. The subtlety is in showing that $j_k$ is surjective.

Recall that $i_U^*$, $i_V^*$, $r_U^*$ and $r_V^*$ are all restriction maps.
If $\alpha \in \Omega^k(X)$ and $i_k(\alpha)=0$,
then the restriction of $\alpha$ to $U$ is zero, as is the restriction to $V$.
But then $\alpha$ itself is the zero form on $X$. This shows that $i_k$ is 
injective. 

Likewise, for any $\alpha \in \Omega^k(X)$, $i_U^*(\alpha)$ and $i_V^*(\alpha)$
agree on $U \cap V$, so $r_U^*i_U^*(\alpha)=r_V^*i_V^*(\alpha)$, so 
$j_k(i_k(\alpha))=0$. Conversely, if $j_k(\beta,\gamma)=0$, then $r_U^*(\beta)=
r_V^*(\gamma)$, so we can stitch $\beta$ and $\gamma$ into a form $\alpha$ on
$X$ that equals $\beta$ on $U$ and equals $\gamma$ on $V$ (and equals both
of them on $U\cap V$), so $(\beta,\gamma) \in Im(i_k)$. 

Now suppose that $\mu \in \Omega^k(U \cap V)$ and that $\{\rho_U, \rho_V\}$
is a partition of unity of $X$ relative to the open cover $\{U,V\}$. Since
the function $\rho_U$ is zero outside of $U$, the form $\rho_U \mu$ 
can be extended to a smooth form on $V$ by declaring that $\rho_U\mu=0$
on $V - U$. 
Note that 
$\rho_U \mu$ is {\bf not} a form on $U$, since $\mu$ is not defined 
on the entire support of $\rho_U$. Rather, $\rho_U \mu$ is a form
on $V$, since $\mu$ is defined at all points of $V$ where $\rho_U \ne 0$.
Likewise, $\rho_V \mu$ is a form on $U$. On $U\cap V$, we have  
$ \mu = \rho_V\mu-(-\rho_U \mu)$. This means that 
$\mu  = j_k(\rho_V \mu,-\rho_U\mu)$. 

The remaining steps are best described in the language of homological algebra.
A {\em cochain complex} $A$ is a sequence of vectors spaces\footnote{For 
(co)homology theories with integer coefficients one usually considers 
sequences of Abelian groups. For de Rham theory, vector spaces will do.}
$\{A^0, A^1, A^2, \ldots\}$ together with maps $d_k: A^k \to A^{k+1}$ such
that $d_{k}\circ d_{k-1}=0$. We also define $A^{-1}=A^{-2}=\cdots$ to 
be 0-dimensional vector spaces (``0'') and 
$d_{-1}=d_{-2}=\cdots$ to be the zero map. 
The {\em $k$-th cohomology of the complex} is
$$ H^k(A) = \frac{\hbox{kernel of $d_k$}}{\hbox{image of $d_{k-1}$}}.$$

A {\em cochain map} $i: A \to B$ between complexes $A$ and $B$ is a family of maps 
$i_k: A^k \to B^k$ such that $d_k(i_k(\alpha))=i_{k+1}(d_k\alpha)$ for all $k$
and all $\alpha \in A^k$, i.e., such that the diagram
$$ \begin{CD}
A^k @>{i_k}>> B^k \\ 
@VV{d_k^A}V @VV{d_k^B}V \\
A^{k+1} @>{i_{k+1}}>> B^{k+1}
\end{CD}
$$
commutes for all $k$, where we have labeled the two differential maps $d_k^A$ and $d_k^B$ to emphasize that they 
are defined on different spaces. 

\smallskip

\noindent{\bf Exercise 8:} Show that a cochain map $i: A \to B$ induces maps in cohomology
$$i_k^\sharp: H^k(A) \to H^k(B); \qquad [\alpha] \mapsto [i_k\alpha].$$

\smallskip

If $A$, $B$ and $C$ are cochain complexes and $i: A \to B$ and $j: B \to C$ 
are cochain maps, then the sequence 
$$ 0 \rightarrow A \xrightarrow i B \xrightarrow j C \rightarrow 0$$
is said to be a {\em short exact sequence} of cochain complexes if, for each
$k$, 
$$ 0 \rightarrow A^k \xrightarrow {i_k} B^k \xrightarrow {j_k} C^k \rightarrow 0$$  
is a short exact sequence of vector spaces.  

So far, we have constructed a short exact sequence of cochain complexes 
with $A^k=\Omega^k(X)$, $B^k=\Omega^k(U) \oplus \Omega^k(V)$, and $C^k=
\Omega^k(U \cap V)$. The theorem on Mayer-Vietoris sequences is then a 
special case of 

\begin{thm}[Snake Lemma] Let $0 \rightarrow A \xrightarrow i B 
\xrightarrow j C \rightarrow 0$ be a short exact sequence of cochain complexes.
Then there is a family of maps $d_k^\sharp: H^k(C) \to H^{k+1}(A)$ such that 
the sequence $$\cdots \to H^k(A) \xrightarrow {i_k^\sharp} H^k(B)   
\xrightarrow {j_k^\sharp} H^k(C) \xrightarrow {d_k^\sharp} H^{k+1}(A)  
\xrightarrow{ i_{k+1}^\sharp} H^{k+1}(B) \rightarrow \cdots$$
is exact. 
\end{thm}

We will use the letters
$\alpha$, $\beta$, $\gamma$, with appropriate subscripts and other
markers, to denote elements of $A$, $B$ and $C$, respectively.  For simplicity, we will 
write ``$d$'' for $d_k^A$, $d_k^B$, $d_k^C$, $d_{k+1}^A$, etc. 
Our first task is to define $d_k^\sharp[\gamma_k]$, where 
$[\gamma_k]$ is a class in $H^k(C)$. 

Since $j_k$ is surjective, we can find $\beta_k \in B^k$ such that 
$\gamma_k=j_k(\beta_k)$. Now, 
$$ j_{k+1}(d\beta_k) = d(j_k(\beta_k))=d(\gamma_k)=0,$$
since $\gamma_k$ was closed. Since $d\beta_k$ is in the kernel of $j_{k+1}$,
it must be in the image of $i_{k+1}$. Let $\alpha_{k+1}$ be such that 
$i_{k+1}(\alpha_{k+1})=d\beta_k$. Furthermore, $\alpha_{k+1}$ is unique, since
$i_{k+1}$ is injective. We define
$$ d_k^\sharp[\gamma_k] = [\alpha_{k+1}].$$

The construction of $[\alpha_{k+1}]$ from $[\gamma_k]$ is summarized in the following diagram:
$$ \begin{CD}
{} @. \beta_k @>{j_k}>> \gamma_k \\
@. @VVdV @. \\
\alpha_{k+1} @>{i_{k+1}}>> d\beta_k @. {}
\end{CD}
$$

For this definition to be well-defined, we must show that 
\begin{itemize} 
\item $\alpha_{k+1}$ is closed. However, 
$$ i_{k+2}(d \alpha_{k+1})  =  d (i_{k+1}\alpha_{k+1})  
= d( d\beta_k) = 0.
$$
Since $i_{k+2}$ is injective, $d \alpha_{k+1}$ must then be zero. 
\item The class $[\alpha_{k+1}]$ does not depend on which $\beta_k$ we chose,
so long as $j_k(\beta_k)=\alpha_k$. The argument in displayed in the diagram
$$ \begin{CD}
(\alpha_k) @. \beta'_k=\beta_k+i_k(\alpha_k) @>{j_k}>> \gamma_k \\
@. @VV{d}V @. \\
\alpha'_{k+1}=\alpha_{k+1}+d \alpha_k @>{i_{k+1}}>> d \beta_k' = d \beta_k +  d i_k\alpha_k @. 
\end{CD}
$$
To see this, suppose that we pick a different $\beta' \in B^k$ with 
$j_k(\beta'_k)=\gamma_k = j_k(\beta_k)$. Then $j_k(\beta'_k-\beta_k)=0$, so 
$\beta'_k-\beta_k$ must be in the image of 
$i_k$, so there exists $\alpha_k \in A^k$ such that $\beta'_k = \beta_k + i_k(\alpha_k)$. But then
$$ d \beta' = d \beta + d (i_k(\alpha_k)) = d \beta + i_{k+1}(d \alpha_k) = i_{k+1}(\alpha_{k+1} + d \alpha_k),$$
so $\alpha'_{k+1} = \alpha_{k+1}+d\alpha_k$. But then $[\alpha'_{k+1}]=[\alpha_{k+1}]$, as
required. 

\item The class $[\alpha_{k+1}]$ does not depend on which cochain $\gamma_k$
we use to represent the class $[\gamma_k]$. 
$$ \begin{CD}
{} @. \beta_{k-1} @>{j_{k-1}}>> \gamma_{k-1} \\
@. @. @.  \\
{} @. \beta_k + d  \beta_{k-1} @>{j_k}>> \gamma_k + d\gamma_{k-1} \\ 
@. @VV{d}V @. \\
\alpha_{k+1} @>{i_{k+1}}>> d \beta_k @. 
\end{CD} $$
Suppose $\gamma'_k = \gamma_k + d
\gamma_{k-1}$ is another representative of the same class. Then
there exists a $\beta_{k-1}$ such that $\gamma_{k-1}=j_{k-1}\beta_{k-1}$. 
But then 
$$
j_k(\beta_k + d\beta_{k-1})  =  
j_k(\beta_k) + j_k(d(\beta_{k-1}))  
=  \gamma_k + d(j_{k-1}\beta_{k-1})  
= \gamma_k + d\gamma_{k-1} = \gamma'.
$$
Thus we can take $\beta'_k=\beta_k+d\beta_{k-1}$. But then $d\beta'_k = d\beta_k$, and our 
cochain $\alpha_{k+1}$ is exactly the same as if we had worked with $\gamma$
instead of $\gamma'$. 
\end{itemize}

Before moving on to the rest of the proof of the Snake Lemma, let's stop and
see how this works for Mayer-Vietoris. 
\begin{enumerate}
\item Start with a class in $H^k(U \cap V)$, represented by a closed
form $\gamma_k\in\Omega^k(U\cap V)$. 
\item Pick $\beta_k=(\rho_V \gamma_k, -\rho_U\gamma_k)$, where $(\rho_U,\rho_V)$
is a partition of unity. 
\item $d\beta_k= (d(\rho_V \gamma_k), -d(\rho_U\gamma_k))$. This is zero outside
of $U \cap V$, since $\rho_V\gamma_k$ and $\rho_U\gamma_k$ were constructed to
be zero outside $U\cap V$. 
\item Since $\rho_U=1-\rho_V$ where both are defined, 
$d\rho_U = - d\rho_V$. Since $d\gamma_k=0$, we then have $-d(\rho_U\gamma_k)=d(\rho_V\gamma_k)$ on $U \cap V$. This means that the forms $d(\rho_V\gamma_k)$ on $U$ and 
$-d(\rho_U\gamma_k)$ on $V$ agree on $U\cap V$, and can be stitched together
to define a closed form $\alpha_{k+1}$ on all of $X$. That is, $d_k^\sharp[\gamma_k]$ is represented by the closed form
$$ d_k^*(\gamma_k) = \begin{cases} d(\rho_V \gamma_k) & \hbox{on $U$} \cr 
-d(\rho_U \gamma_k) & \hbox{on $V$,} \end{cases} $$
and the two definitions agree on $U \cap V$. 
\end{enumerate}

Returning to the Snake Lemma, we must show six inclusions:
\begin{itemize}

\item $Im(i_k^\sharp)\subset Ker(j_k^\sharp)$, i.e. that 
$j_k^\sharp i_k^\sharp[\alpha_k] = 0$ for any closed cochain $\alpha_k$.
This follows from 
$$j_k^\sharp(i_k^\sharp[\alpha_k]) = j_k^\sharp[i_k \alpha_k]=
[j_k(i_k(\alpha_k))]=0,$$
since $j_k \circ i_k=0$. 

\item $Im(j_k^\sharp) \subset Ker(d_k^\sharp)$, i.e. that 
$d_k^\sharp \circ j_k^\sharp=0$. If $[\beta_k]$ is a class in $H^k(B)$, 
then $j_k^\sharp[\beta_k]=[j_k\beta_k]$. To apply $d_k^\sharp$ to this, we must\\
(a) find a cochain in $B^k$ that maps to $j_k\beta_k$. Just take $\beta_k$ itself! \\ 
(b) Take $d$ of this cochain. That gives 0, since $\beta_k$ is closed. \\
(c) Find an $\alpha_{k+1}$ that maps to this. This is $\alpha_{k+1}=0$. 

\item $Im(d_{k}^\sharp)\subset Ker(i_{k+1}^\sharp)$. 
If $\alpha_{k+1}$ represents $d_k^\sharp[\gamma_k]$, then $i_{k+1}\alpha_{k+1}
= d\beta_k$ is exact, so $i_{k+1}^\sharp[\alpha_k]=0$. 

\item $Ker(j_k^\sharp) \subset Im(i_k^\sharp)$. If $[j_k \beta_k]=0$, then
$j_k \beta = d(\gamma_{k-1})$. Since $j_{k-1}$ is surjective, we can find
a $\beta_{k-1}$ such that $j_{k-1}\beta_{k-1}= \gamma_{k-1}$. But then 
$j_k(d\beta_{k-1})=d(j_{k-1}(\beta_{k-1}))=d\gamma_{k-1}=\gamma_k$. Since 
$j_k(\beta_k-d\beta_{k-1})=0$, we must have $\beta_k-d\beta_{k-1} = 
i_k(\alpha_k)$ for some $\alpha_k \in A^k$. Note that $i_{k+1}d\alpha_k =
d(i_k \alpha_k)=d\beta_k - d(d\beta_{k-1}) = 0$, since $\beta_k$ is closed. 
But $i_{k+1}$ is injective, so $d\alpha_k$ must be zero, so $\alpha_k$
represents a class $[\alpha_k] \in H^k(A)$. But then 
$i_k^\sharp[\alpha_k] = [i_k(\alpha_k)] = [\beta_k + d\beta_{k-1}]=[\beta_k]$,
so $[\beta_k]$ is in the image of $i_k^\sharp$. 

\item $Ker(d_k^\sharp) \subset Im(j_k^\sharp)$. Suppose that 
$d_k^\sharp[\gamma_k]=0$. This means that the $\alpha_{k+1}$ constructed to 
satisfy $i_{k+1}\alpha_{k+1} = d\beta_k$, where $\gamma_k=j_k(\beta_k)$, must 
be exact. That is, $\alpha_{k+1}=d\alpha_k$ for some $\alpha_k \in A^k$.
But then
$$ d\beta_k = i_{k+1}(\alpha_{k+1}) = i_{k+1}(d\alpha_k)=d(i_k(\alpha_k)).$$
Thus $\beta_k - i_k(\alpha_k)$ must be closed, and must represent a 
class in $H^k(B)$. But 
$$j_k^\sharp[\beta_k-i_k(\alpha_k)] = [j_k\beta_k - j_k(i_k(\alpha_k))]
= [j_k(\beta_k)] = [\gamma_k],$$
since $j_k\circ i_k = 0$. Thus $[\gamma_k]$ is in the image of $j_k^\sharp$.

\item $Ker(i_{k+1}^\sharp)\subset Im(d_k^\sharp)$. Let $\alpha_{k+1}$ be a closed
cochain in $A^{k+1}$, and suppose that $i_{k+1}^\sharp[\alpha_{k+1}]=0$.
This means that $i_{k+1}\alpha_{k+1}$ is exact, and we can find a 
$\beta_k \in B^k$ such that $i_{k+1}\alpha_{k+1} = d\beta_k$. Note that 
$d  j_k \beta_k = j_{k+1}d\beta_k = j_{k+1}(i_{k+1}(\alpha_{k+1}))=0$, so 
$j_k\beta_k$ is closed. But then $[j_k \beta_k] \in H^k(C)$ and 
$[\alpha_{k+1}] = d_k^\sharp[j_k\beta_k]$ is in the image of $d_k^\sharp$. 
\end{itemize}

\noindent{\bf VERY IMPORTANT Exercise 9:} Each of these arguments is called a 
``diagram chase'', in that it involves using properties at one spot of a 
commutative diagram to derive properties at an adjacent spot. 
For each of these arguments, draw an appropriate diagram to illustrate what is going on. (In the history of 
the universe, nobody has fully understood the Snake Lemma without drawing 
out the diagrams himself or herself.)
   
\section{Using Mayer-Vietoris}

Our first application of Mayer-Vietoris will be to determine $H^k(S^n)$  for all $k$ and $n$. We begin with the 1-sphere, whose
cohomology we already computed using other methods. Here we'll compute it in detail using Mayer-Vietoris, as an example of 
how the machinery works. 

Let $S^1$ be the unit circle embedded in $\R^2$. Let $U=\{ (x,y) \in S^1 | y < 1/2\}$ and let $V = \{ (x,y) \in S^1 | y > -1/2\}$. 
The open sets $U$ and $V$ are both diffeomorphic to $\R$, so $H^0(U)=H^0(V)=\R$ and $H^k(U)=H^k(V)=0$ for $k>0$. 
$U \cap V$ consists of two intervals, one with $x>0$ and the other with $x<0$. Each is diffeomorphic to the line, so 
$H^0(U \cap V) = \R^2$ and $H^k(U \cap V)= 0$ for $k>0$. The Mayer-Vietoris sequence
$$ 0 \to H^0(S^1) \to H^0(U) \oplus H^0(V) \to H^0(U \cap V) \to H^1(S^1) \to H^1(U)\oplus H^1(V) \to \cdots$$
simplifies to 
$$ 0 \to H^0(S^1) \xrightarrow {i_0^\sharp} \R^2 \xrightarrow {j_0^\sharp} \R^2 \xrightarrow {d_0^\sharp}
H^1(S^1) \xrightarrow {j_1^\sharp} 0 $$
\begin{enumerate}
\item Since $S^1$ is connected, $H^0(S^1)=\R$. 
\item Since $i_0^\sharp$ is injective, the image of $i_0^\sharp$ is 1-dimensional. 
\item This makes the kernel of $j_0^\sharp$ 1-dimensional, so $j_0^\sharp$ has rank 1. 
\item This makes the kernel of $d_0^\sharp$ 1-dimensional, so $d_0^\sharp$ has rank 1. 
\item This make $H^1(S^1)$ 1-dimensional, and we conclude that $H^1(S^1)=H^0(S^1)=\R$ (and $H^k(S^1)=0$ for
$k>1$ since $S^1$ is only 1-dimensional).
\end{enumerate}

What's more, we can use Mayer-Vietoris to find a generator of $H^1(S^1)$. We just take an element of $H^0(U\cap V)$
that is not in the image of $j_0^\sharp$ and apply $d_0^\sharp$ to it.   In fact, let's see what all the maps in the Mayer-Vietoris
sequence are.

$H^0(S^1)$ and $H^0(U)$ and $H^0(V)$ are each generated by the constant function 1. Restricting 1 from $X$ to $U$ or from 
$X$ to $V$ gives a function that is 1 on $U$ (or $V$), and restricting 1 from $U$ to $U\cap V$, or from $V$ to $U \cap V$, gives a function that is 1 on both components of $U \cap V$. Thus 
$$i_0^\sharp(s)  = (s,s) \qquad j_0^\sharp(s,t)  =  (s-t, s-t).$$
Now consider $(1,0) \in H^0(U \cap V)$. This is the function $\gamma_0$ that is 1 on one 
component of $U\cap V$ (say, the piece with
$x>0$) and 0 on the other component. Now let $\rho_U$ be a smooth function that is 1 for $y<-1/10$ and is 0 for $y>1/10$, and 
let $\rho_V = 1-\rho_U$. Then $\rho_V \gamma_0$ is a function on $U$ that is 
\begin{itemize}
\item Equal to $\rho_V$ on the part of $U \cap V$ with $x>0$. 
\item Equal to 0 on the part of $U \cap V$ with $x < 0$. 
\item Equal to 0 on the rest of $U$ since there $\rho_V=0$.
\end{itemize}
Similarly, $-\rho_U \gamma$ is a function on $V$ that is only nonzero on the part of $U\cap V$ where $x>0$.  We then
have $\alpha_1$ is a 1-form that is 
\begin{itemize}
\item Equal to $d\rho_V = -d\rho_U$ on the part of $U \cap V$ with $x>0$, and 
\item Equal to $0$ everywhere else. 
\end{itemize}
Since $\rho_V$ increases from 0 to 1 as we move counterclockwise along this interval, $\int_{S^1} \alpha_1 = 1$.  
In fact, the support of 
$\alpha_1$ is only a small part of $U\cap V$, and is included in the region where $x>0$ and $-1/10 \le y \le 1/10$. This is  
called a ``bump form'', in analogy with bump functions. 

 Now we compute the cohomology of the $n$-sphere $S^n$. 
 Let $S^n$ be the unit sphere in $\R^{n+1}$ and let $U$ and $V$ be the portions of that sphere with $x_{n+1} <1/2$
 and with $x_{n+1} > -1/2$. Each of $U$ and $V$ is diffeomorphic to $\R^n$, so $H^k(U)=H^k(V)=\R$ for $k=0$ and 
 $H^k(U)=H^k(V)=0$ otherwise. $U\cap V$ is a strip around the equator and is diffeomorphic to $\R \times S^{k-1}$,
 and so has the same cohomology as $S^{n-1}$.  Since $H^k(U)=H^k(V) =0$ for $k>0$, the sequence  
 $$  H^k(U)\oplus H^k(V) \to H^k(U\cap V) \to H^{k+1}(S^n) \to H^{k+1}(U) \oplus H^{k+1}(V) $$
 is 
 $$ 0 \to H^k(S^{n-1}) \to H^{k+1}(S^n) \to 0,$$
 so $H^{k+1}(S^n)$ is isomorphic to $H^k(S^{n-1})$.  By induction on $n$, this shows that $H^k(S^n)=\R$ when $k=n$ and
 is 0 when $0<k \ne n$. Furthermore, the generator of $H^n(S^n)$ can be realized as a bump form, equal to 
 $d \rho_V$ wedged with the generator of $H^{n-1}(S^{n-1})$, which in turn is a bump 1-form wedged with the 
 generator of $H^{n-1}(S^{n-2})$. Combining steps, this gives a bump $n$-form of total integral 1, localized near the point
 $(1,0,0,\ldots,0)$.    
 
 \smallskip
 
 \noindent {\bf Exercise 10:}  Let $T = S^1 \times S^1$ be the 2-torus. By dividing one of the $S^1$ factors circle two (overlapping) 
 open sets, we can divide $T$ into two cylinders $U$ and $V$, such that $U\cap V$ is itself the disjoint union of two cylinders. 
 Use this partition and the Mayer-Vietoris sequence to compute the cohomology of $X$. Warning:  unlike with the circle, 
 the dimensions of $H^k(U)$, $H^k(V)$ and $H^k(U\cap V)$ are not enough to solve this problem. You have to actually 
 study what $i_k^\sharp$, $j_k^\sharp$ and/or $d_k^\sharp$ are doing. [Note:  $H^1(\R \times S^1)=\R$, by integration
 over the fiber. This theorem also implies that the generator of $H^1(\R \times S^1)$ is just the pullback of a generator of $H^1(S^1)$ to  $\R \times S^1$. You should be able to explicitly write generators for $H^1(U)$, $H^1(V)$ and $H^1(U\cap V)$
 and see how the maps $i_1^\sharp$ and $j_1^\sharp$ behave.]
 
 \smallskip
 
 \noindent {\bf Exercise 11:} Let $K$ be a Klein bottle. Find open sets $U$ and $V$ such that $U \cup V = K$, such that 
 $U$ and $V$ are cylinders, and such that $U \cap V$ is the disjoint union of two cylinders. In other words, the exact same
 data as with the torus $T$. The difference is in the ranks of some of the maps. Use Mayer-Vietoris to compute the cohomology
 of $K$. 
 
 \section{Good covers and the Mayer-Vietoris argument}
 
 A set is {\em contractible} if it deformation retracts to a single point, in which case it has the same cohomology as a single 
 point, namely $H^0=\R$ and $H^k=0$ for $k>0$. In the context of $n$-manifolds, an open contractible set is something
 diffeomorphic to $\R^n$. 
 
 \smallskip
 
 \noindent{\bf Exercise 12:} Suppose that we are on a manifold with a Riemannian metric, so that there is well-defined notion
 of geodesics. Suppose furthermore that we are working on a region on which geodesics are unique: there is one and only one geodesic from a point $p$ to a point $q$. An open submanifold $A$ is called {\em convex} if, for any two points $p,q \in A$, the geodesic from $p$ to $q$ is entirely in $A$. Show that a convex submanifold is contractible. 
 
 \smallskip
 
 \noindent {\bf Exercise 13:} Show that the (non-empty) intersection of any collection of convex sets is convex, and hence 
 contractible. 
 
 \medskip
 
 A {\em good cover} of a topological space is an open cover $\{U_i\}$ such that any finite intersection
 $U_{i_1} \cap U_{i_1} \cap \cdots \cap U_{i_k}$ is either empty or contractible. 
 
 \begin{thm} Every compact manifold $X$ admits a finite good cover.  \end{thm}
 
 \begin{proof} First suppose that $X$ has a Riemannian metric. Then each point has a convex geodesic (open) 
 neighborhood. Since the 
 intersection of two (or more) convex sets is convex, these neighborhoods form a good cover. Since $X$ is compact, there
 is a finite sub-cover, which is still good. 
 
If $X$ is embedded in $\R^N$, then the Riemannian structure comes from the embedding. 
The only question is how to get a Riemannian metric when $X$ is an {\em abstract} manifold. To do this, partition $X$ into
coordinate patches. Use the Riemannian metric on $\R^n$ on each patch. Then stitch them together using partitions
of unity. Since any positive linear combination of inner products still satisfies the axioms of an inner product, this gives a 
Riemannian metric for $X$.
\end{proof}

We next consider how big the cohomology of a manifold $X$ can be. $H^k(X)$ is the quotient of two infinite-dimensional vector
spaces. Can the quotient be infinite-dimensional? 

If $X$ is not compact, it certainly can. For example, consider the connected sum of an infinite sequence of tori. $H^1$ of such
a space would be infinite-dimensional. However, 

\begin{thm} If $X$ is a compact $n$-manifold, then each $H^k(X)$ is finite-dimensional. \end{thm}

\begin{proof} The proof is by induction on the number of sets in a good cover. 

\begin{itemize}
\item If a manifold $X$ admits a good cover with a single set $U_1$, then $X$ is either contractible or empty, so $H^0(X)=\R$ or 0 and all other cohomology groups are trivial. 
\item Now suppose that all manifolds (compact or not) that admit open covers with at most $m$ elements have finite-dimensional 
cohomology, and suppose that $X$ admits a good cover $\{U_1, \ldots, U_{m+1}\}$. Let $U=U_1 \cup \cdots U_m$ and
let $V=U_{m+1}$.  But then $\{U_1\cap U_{m+1}, U_2\cap U_{m+1}, \cdots, 
U_m\cap U_{m+1}$ is a good cover for $U\cap V$ with $m$ elements, so the cohomologies of $U$, $V$ and $U\cap V$ 
are finite-dimensional.
\item The Mayer-Vietoris sequence says that $H^{k-1}(U \cap V) \xrightarrow {d_{k-1}^\sharp} H^{k}(X) \xrightarrow {i_{k}^\sharp} H^k(U) \oplus H^k(V)$ is exact. However, $H^{k-1}(U\cap V)$ and $H^k(U) \oplus H^k(V)$ are finite 
dimensional, since $U\cap V$, $U$ and $V$ all admit good covers with at most $m$ elements. Thus $H^k(X)$ must also 
be finite-dimensional. 
\item By induction, all manifolds with finite good covers have finite-dimensional cohomologies.
\item Since all compact manifolds have finite good covers, all compact manifolds have finite-dimensional cohomologies.
\end{itemize}
\end{proof}

This proof was an example of the {\em Mayer-Vietoris argument}. In general, we might want to prove that all 
spaces with finite good covers have a certain property $P$. Examples of such properties include finite-dimensional cohomology, 
Poincare duality (between de Rham cohomology and something called ``compactly supported cohomology''), the Kunneth formula for cohomologies of product spaces, and the isomorphism between de Rham cohomology and singular cohomology with 
real coefficients. The steps of the argument are the same for all of these theorems:
\begin{enumerate}
\item Show that every contractible set has property $P$. 
\item Using the Mayer-Vietoris sequence, show that if $U$, $V$, and $U\cap V$ have property $P$, then so does
 $U \cup V$. 
 \item Proceeding by induction as above, showing that all manifolds with finite good covers have property $P$. 
\item Conclude that all compact manifolds have property $P$. 
\end{enumerate}
For lots of examples of the Mayer-Vietoris principle in action, see Bott and Tu's excellent book {\em Differential forms in algebraic
topology}.

\chapter{Top cohomology, Poincare duality, and degree}

\section{Compactly supported cohomology}

Integration is a pairing between compactly supported forms and
oriented manifolds. Given an oriented manifold $X$ and a compactly 
supported $n$-form
$\omega$, we compute $\int_X\omega$. Of course, if $X$ is compact,
then {\em every} form on $X$ is compactly supported. Also, if $X$ is
not compact and $\omega$ is not compactly supported, then we can often
compute $\int_X \omega$ via limits.  But at its core, integration is
about compactly supported forms.

As such, it makes sense to analyze the cohomology of compactly
supported forms. Let $\Omega_c^k(X)$ denote the vector space of
compactly supported $k$-forms on the $n$-manifold $X$. Since $d$ of a
compactly supported form is compactly supported, we have a complex:
$$ 0 \to \Omega_c^1(X) \xrightarrow d \Omega_c^2(X) \xrightarrow d \Omega_c^3(X) \xrightarrow d \cdots 
\xrightarrow d \Omega_c^n(X) \to 0,$$ and we define $H_c^k(X)$ to be
the $k$-th cohomology of this complex. That is,
$$ H^k_c(X) = \frac{\hbox{Closed, compactly supported $k$-forms on }X}{d\hbox{(Compactly supported $(k-1)$-forms on }X)}.$$
As with ordinary (de Rham) cohomology, we start by studying the
cohomology of $\R^n$.

\begin{thm} $H^k_c(\R^n)=\R$ if $k=n$ and 0 otherwise. \end{thm}

\begin{proof} I will treat the cases $n=0$, $n=1$ and $n=2$ by hand,
  and then show how to get all larger values of $n$ by induction.

If $n=0$, then $\R^n$ is compact, and $H^0_c(\R^0)=H^0(\R^0)=\R$. 

If $n=1$, then $H^0_c$ consists of compactly supported constant
functions. But a constant function is only compactly supported if the
constant is zero! Thus $H^0_c(\R)=0$. As for $H^1_c$, suppose that
$\alpha$ is a 1-form supported on the interval $[-R,R]$. Let $f(x) =
\int_{-R}^x \alpha$. Then $\alpha =df$, and $f$ is the only
antiderivative of $\alpha$ that is zero for $x<-R$.
Meanwhile, $f(x)=0$ for $x>R$ if and only if $\int_\R \alpha =0$, so 
$\alpha$ is $d$ of a compactly supported function if and
only if $\int_\R \alpha = 0$, and
$$ H^1_c(\R) = \frac{\hbox{All compactly supported 1-forms on $\R$}}{\hbox{Compactly supported 1-forms on $\R$
with integral zero.}} = \R.$$

If $n=2$, then $H^0_c(\R^2)$ is trivial, since the only compactly
supported constant function is zero. We also have that
$H^1_c(\R^2)=0$, since if $\alpha$ is a closed 1-form supported on a
(closed subset of a) ball of radius $R$ around the origin, and if $p$
is a point outside that ball, then $\alpha = df$, where $f(x) =
\int_p^x \alpha$. (This integral doesn't depend on the path chosen
from $p$ to $x$ because $\alpha$ is closed and $\R^2$ is simply
connected.)  The function $f$ is supported on the ball of radius $R$,
since if $|x|>R$, then there is a path from $p$ to $x$ that avoids the
support of $\alpha$ altogether.

The tricky thing is showing that a compactly supported 2-form $\beta$
on $\R^2$ is $d$ of a compactly supported 1-form if and only if
$\int_{\R^2} \beta = 0$.

The ``only if'' part is just Stokes' Theorem. If $\beta = d\gamma$,
with $\gamma$ compactly supported, then $\int_{\R^2} \beta =
\int_{\R^2} d\gamma = \int_{\partial\R^2}\gamma= 0$.

To prove the other implication, suppose that $\beta = b(x,y) dx \wedge
dy$ is a compactly supported 2-form of total integral zero, say
supported on a closed subset of the square $[-R,R] \times [-R,R]$ for
some $R>1$.  We define a number of useful functions and forms as
follows:
\begin{itemize}
\item Let $f(s)$ be a smooth function of $\R$ with $f(s)=1$ for $s \ge
  1$ and $f(s)= 0$ for $s \le 0$. Then $df = f'(s) ds$ is a bump 1-form,
  supported on $[0,1]$, of total integral 1.
\item Let $B(x) = \int_{-R}^R b(x,y) dy$. Note that $B$ is compactly
  supported and that $\int_\R B dx = 0$. By our 1-dimensional analysis,
there is a
  compactly supported function $G(x)$ such that $dG=B(x)dx$.
\item Let $\tilde \beta =B(x) f'(y) dx \wedge dy$. Note that this is 
$d$ of $G(x) f'(y) dy$, which in turn is a compactly supported
1-form. 
\item Now let $C(x,y) = \int_{-R}^y (b(x,s) - B(x) f'(s)) ds$. This is 
compactly supported since $\int_{-R}^R (b(x,s) -f'(s)B(x)) ds = 0$. 
\item $d (-C(x,y) dx) = \frac{\partial C}{\partial y} dx \wedge dy 
= (b(x,y)-B(x)f'(y)) dx \wedge dy = \beta - \tilde \beta$.  
\item Since both $\tilde \beta$ and $\beta - \tilde \beta$ can be written 
as $d$ of a compactly supported 1-form, so can $\beta$. 
\end{itemize}

To go beyond $n=2$, we need a variant on the integration-over-a-fiber
argument that we previously used to get the Poincare Lemma. 
We want to compare the cohomologies of $X$ and $X \times \R$.\footnote{In this
construction we work with $X \times \R$ rather than $\R \times X$ to simplify
the signs in some of our computations.} We will
construct maps
$$i_k: \Omega^k_c(X) \to \Omega^{k+1}_c(X \times \R); \qquad j_{k+1}: \Omega_c^{k+1}(X \times \R) \to \Omega^k_c(X)$$
and a {\em homotopy operator} $P_{k+1}: \Omega^{k+1}_c(X \times \R) \to \Omega^k_c(X \times \R)$ such that 
\begin{eqnarray}
d \circ i &=&  i \circ d \cr
d \circ j & = &  j \circ d \cr 
j \circ i & = & 1 \cr 
(1-i\circ j) &=& \pm dP  \pm Pd,
\end{eqnarray}
where we have suppressed the subscripts on $i_k$, $j_k$, $d_k$, and
$P_k$ and the identity map 1, and where the signs in the last equation
may depend on $k$.  The first line implies that $i$ induces a map
$i^\sharp: H^k_c(X) \to H^{k+1}_c(X \times \R)$, and the second that
$j$ induces a map $j^\sharp: H^{k+1}_c(X \times \R) \to H^k_c(X)$. The
third line implies that $j^\sharp \circ i^\sharp$ is the identity, and
the fourth implies that $i^\sharp \circ j^\sharp$ is also the
identity. Thus $i^\sharp$ and $j^\sharp$ are isomorphisms, and
$H^{k+1}_c(X \times \R) = H^k_c(X)$. In particular,
$H^{k+1}_c(\R^{n+1}) = H^k_c(\R^n)$, providing the inductive step of
the proof of our theorem.

If $\alpha = \sum \alpha_I(x) dx^I$ is a compactly supported $k$-form
on $X$, let 
$$i(\alpha) = \sum_I \alpha_I(x) f'(s) dx^I \wedge ds =
(-1)^k \sum_I f'(s) \alpha_I(x) ds \wedge dx^I,$$ 
where $f'(s)ds$ is a
bump form on $\R$ with integral 1.  Let $\phi$ be the pullback of this
bump form to $X \times \R$.  Another way of writing the formula for
$i$ is then
$$i(\alpha) = \pi_1^*(\alpha) \wedge \phi,$$
where $\pi_1$ is the natural projection from $X \times \R$ to $X$. We check that 
$$i(d\alpha) = \pi_1^*(d\alpha) \wedge \phi = d(\pi_1^*(\alpha) \wedge \phi),$$
since $d \phi =0$. 

Next we define $j$. Every compactly supported $k$-form $\alpha$ on $X \times \R$ can be written as a sum of two pieces:
$$ \alpha = \sum_I \alpha_I(x,s) dx^I + \sum_J \gamma_J(x,s) dx^J \wedge ds,$$
where each $I$ is a $k$-index and each $J$ is a $(k-1)$-index. We define
$$j(\alpha) = \sum_J \left ( \int_{-\infty}^\infty \gamma(x,s) ds \right ) dx^J.$$
Note that 
\begin{eqnarray} 
j(d\alpha) & = & j\left  (\sum_{I,j} \partial_j \alpha_I(x,s) dx^j \wedge dx^I + \sum_I \partial_s \alpha_I(x,s) ds \wedge dx^I
+ \sum_{j,J} \partial_j \gamma_J(x,s) dx^j \wedge dx^J \wedge ds \right )\cr 
& = & (-1)^k \sum_I \left ( \int_{-infty}^\infty \partial_s \alpha_I(x,s) ds \right ) dx^I + \sum_{j,J}
\left ( \int_{-\infty}^\infty \partial_j \gamma_J(x,s)\right ) dx^j \wedge dx^J \cr 
&=& 0 + \sum_{j,J} \partial_j \left ( \int_{-\infty}^\infty \gamma_J(x,s) \right ) dx^j \wedge dx^J \cr 
& = & d(j(\alpha)),
\end{eqnarray}
where we have used the fact that $\alpha_I(x,s)$ is compactly
supported, so $\int_{-\infty}^\infty \partial_s \alpha_I(x,s) ds =
\alpha_I(x,\infty)-\alpha_I(x,-\infty)=0.$ In the computation of
$H^2_c(\R^2)$, the form $\tilde \beta$ was precisely $i\circ
j(\beta)$.

Now, for $\alpha \in \Omega^k_c(X \times \R)$, let
$$ P(\alpha)(x,s) = \sum_J \left ( \int_{-\infty}^s \gamma_J(x,t) dt  - f(s) \int_{-\infty}^\infty \gamma_J(x,t) dt \right ) dx^J.$$ 
This gives a compactly supported form, since for $s$ large and
positive the two terms cancel, while for $s$ large and negative both
terms are zero.

\smallskip

\noindent {\bf Exercise 1:} For an arbitrary form $\alpha \in
\Omega^{k}_c(X \times \R)$, compute $d(P( \alpha))$ and $P(d\alpha)$,
and show that $\alpha - i(j(\alpha)) = \pm dP(\alpha) \pm P(d\alpha)$.

\smallskip

\end{proof}

\section{Computing the top cohomology of compact manifolds}

Having established the basic properties of compactly supported forms
on $\R^n$, and hence compactly supported forms on a coordinate patch,
we consider $H^n(X)$.

\begin{thm} Let $X$ be a compact, connected $n$-manifold. 
$H^n(X) = \R$ if $X$ is orientable and is 0 if $X$ is not orientable. \end{thm}

\begin{proof}
We prove the theorem in three steps:
\begin{enumerate}
\item Showing that, if $X$ is oriented, then $H^n(X)$ is at least 1-dimensional.
\item Showing that, regardless of orientation, $H^n(X)$ is at most 1-dimensional. 
\item Showing that, if $X$ is not oriented, that $H^n(X)$ is trivial. 
\end{enumerate} 

For the first step, suppose $\alpha \in \Omega^n(X)$. If $\alpha =
d\beta$ is exact, then by Stokes' Theorem, $\int_X \alpha = \int_X d
\beta = \int_{\partial X} \beta = 0$, since $\partial X$ is the empty
set.  Since every exact $n$-form integrates to zero, a closed form
that doesn't integrate to zero must represent an non-trivial class in
$H^n$. However, such forms are easy to generate. Pick a point $p$ and
a coordinate patch $U$, and take a bump form of total integral 1
supported on $U$.

For the second and third steps, we will show that an arbitrary $n$-form $\alpha$
is {\em cohomologous} to a finite sum of bump forms, where two closed
forms are said to be cohomologous if they differ by an exact
form. (That is, if they represent the same class in cohomology.) We then
show that any bump form is cohomologous to a multiple of a specific
bump form, which then generates $H^n(X)$. If $X$ is not orientable, we
will then show that this generator is cohomologous to minus
itself, and hence is exact.

Pick a partition of unity $\{ \rho_i\}$ subordinate to a collection of
coordinate patches. Any $n$-form $\alpha$ can then be written as
$\sum_i \rho_i \alpha_i$. Since $X$ is compact, this is a finite sum.
Now suppose that $\alpha_i = \rho_i \alpha$ is compactly supported in
the image of a parametrization $\psi_i: U_i \to X$.  Let $\phi_i$ be a
bump form on $U_i$ of total integral 1 localized around a point $a_i$,
and let $c_i = \int_{U_i} \psi_i^* \alpha_i$, where we are using the
canonical orientation of $\R^n$.  Then $\psi_i^*\alpha_i - c_i \phi_i$
is $d$ of a compactly supported $(n-1)$-form on $U_i$.  This implies
that $\alpha_i - c_i (\psi_i^{-1})^* \phi_i$ is $d$ of an $(n-1)$-form
that is compactly supported on $\psi_i(U_i)$, and can thus be extended
(by zero) to be a form on all of $X$. Thus $\alpha_i$ is cohomologous
to $c_i (\psi_i^{-1})^*\phi_i$.  That is, to a bump form localized
near $p_i = \psi_i(a_i)$.

The choice of $a_i$ was arbitrary, and the precise formula for the
bump form was arbitrary, so bump forms with the same integral
supported near different points of the same coordinate patch are
always cohomologous. However, this means that if $\phi$ and $\phi'$ are
bump $n$-forms supported
near {\em any} two points $p$ and $p'$ of $X$, then $\phi$ is cohomologous
to a multiple of $\phi'$. We
just find a sequence of coordinate patches $V_i=\psi_i(U_i)$ and
points $p_0=p \in V_1$, $p_1 \in V_1 \cap V_2$, $p_2 \in V_2 \cap
V_3$, etc. A bump form near $p_0$ is cohomologous to a multiple of a 
bump form near
$p_1$ since both points are in $V_1$. But that is cohomologous to a
multiple of a bump form near $p_2$ since both $p_1$ and $p_2$ are in $V_2$, etc.
Thus $\alpha$ is cohomologous to a sum of bump forms, which is in turn
cohomologous to a single bump form centered at an arbitrarily chosen
location. This shows that $H^n(X)$ is at most 1-dimensional.

Now suppose that $X$ is not orientable. Then we can find a sequence of
coordinate neighborhoods, $V_1,\ldots, V_N$ with $V_N=V_1$, such that
there are an odd number of orientation reversing transition
functions. Starting with a bump form at $p$, we apply the argument of
the previous paragraph, keeping track of integrals and signs, and 
ending up with a
bump form at $p'=p$ that is {\em minus} the original bump form. Thus twice
the bump form is cohomologous to zero, and the bump form itself is
cohomologous to zero. Since this form generated $H^n(X)$, $H^n(X)=0$.

\end{proof}

The upshot of this theorem is not only a calculation of $H^n(X)=\R$
when $X$ is connected, compact and oriented, but the derivation of a
generator of $H^n(X)$, namely {\em any} form whose total integral is
1. This can be a bump form, it can be a multiple of the
(n-dimensional) volume form, and there are infinitely many other
possibilities. The important thing is that integration gives an
isomorphism
$$ \int_X: H^n(X) \to \R.$$

\section{Poincare duality}

Suppose that $X$ is an oriented $n$-manifold (not
necessarily compact and not necessarily connected), 
that $\alpha$ is a closed, compactly supported
$k$-form, and that $\beta$ is a closed $n-k$ form.  Then $\alpha
\wedge \beta$ is compactly supported, insofar as $\alpha$ is compactly
supported.  Integration gives a map
\begin{eqnarray} \int_X: H^k_c(X) \times H^{n-k}(X) & \to & \R \cr 
[\alpha]\times[\beta] & \mapsto & \int_X \alpha \wedge \beta.
\end{eqnarray}

\smallskip

\noindent{\bf Exercise 2:} Suppose that $\alpha$ and $\alpha'$ represent
the same class in $H^k_c(X)$ and that $\beta$ and $\beta'$ represent
the same class in $H^{n-k}(X)$. Show that $\int_X \alpha' \wedge
\beta' = \int_X \alpha \wedge \beta$.

\smallskip

This implies that integration gives a map from $H^{n-k}(X)$ to the
dual space of $H^k_c(X)$.

\begin{thm}[Poincare duality] If $X$ is an orientable manifold that
  admits a finite good cover, then integration gives an isomorphism
  between $H^{n-k}(X))$ and $(H^k_c(X))^*$.
\end{thm}

\begin{proof} A complete proof is beyond the scope of these
  notes. A complete presentation can be found in Bott and Tu.  Here is
  a sketch.
\begin{itemize}
\item The theorem is true for a single coordinate patch, since
  $H^{n-k}(\R^n)$ and $H^k_c(\R^n)^*$ are both $\R$ when $k=n$ and 0
  otherwise, with integration relating the two as above.
\item We construct a Mayer-Vietoris sequence for compactly supported cohomology. With regard to incusions, 
this runs the opposite direction 
as the usual Mayer-Vietoris sequence:
$$ \cdots \rightarrow H^k_c(U\cap V) \rightarrow H^k_c(U) \oplus H^k_c(V) \rightarrow H^k_c(U \cup V) \rightarrow H^{k+1}(U \cap V) \rightarrow \cdots$$
\item Looking at dual spaces,  we obtain an exact sequence 
$$ \cdots \rightarrow H^k_c(U \cup V)^* \rightarrow H^k_c(U)^* \oplus H^k_c(V)^* \rightarrow H^k_c(U\cap V)^* 
\rightarrow H^{k-1}_c(U \cup V)^* \rightarrow \cdots $$ With respect
to inclusions, this goes in the same direction as the usual
Mayer-Vietoris sequence (going from the cohomology of $U\cup V$ to
that of $U$ and $V$ to that of $U \cap V$ to that of $U \cup V$, etc.), only
with the index $k$ decreasing at the $U \cap V \to U \cup V$ stage
instead of increasing. However, this is precisely what we need for Poincare
duality, since decreasing the dimension $k$ is the same thing as
increasing the codimension $n-k$.
\item The {\em five lemma} in homological algebra says that if you have a commutative diagram
$$ \begin{CD}
A @>>> B @>>> C @>>> D @>>> E \cr  
@VV\alpha V @VV\beta V @VV\gamma V @VV\delta V @VV\epsilon V \cr 
A' @>>> B' @>>> C' @>>> D' @>>> E' 
\end{CD}
$$
with the rows exact, and if $\alpha$, $\beta$, $\delta$ and $\epsilon$
are isomorphisms, then so is $\gamma$. Comparing the Mayer-Vietoris
sequences for $H^{n-k}$ and for $(H^k_c)^*$, the five lemma says that,
if for all $k$ integration gives isomorphisms between $H^{n-k}(U)$ and
$H^k_c(U)^*$, between $H^{n-k}(V)$ and $H^k_c(V)^*$, and between
$H^{n-k}(U\cap V)$ and $H^k_c(U\cap V)^*$, then
integration also induces isomorphisms between $H^{n-k}(U\cup V)$ and
$H^k_c(U\cup V)^*$.
 \item We then use the Mayer-Vietoris argument, proceeding by 
induction on the cardinality of a good cover. 
 \end{itemize}
 \end{proof}
 
\smallskip

 \noindent{\bf Not-so-hard Exercise 3:} Prove the Five Lemma.

\smallskip

\noindent{\bf Somewhat harder Exercise 4:} Set up the Mayer-Vietoris sequence 
for compactly supported cohomology, using the fact that a compactly supported
form on $U\cap V$ can be extended by zero to give a compactly supported form
on $U$ or on $V$, thus defining a map $i: \Omega^k_c(U \cap V) \to 
\Omega^k_c(U)\oplus \Omega^k_c(V)$, and that we can similarly define a map
$j: \Omega^k_c(U)\oplus \Omega^k_c(V) \to \Omega^k(U\cup V)$. 

\smallskip

 \noindent{\bf Much harder Exercise 5:} Fill in the details of the proof of
 Poincare duality. The main effort is in setting up the two Mayer-Vietoris
sequences and showing that connecting them by integration gives a 
commutative diagram. You may have to tweak the signs of some of the maps 
to make the diagram commutative. Or you can prove a souped up version
of the five lemma in which every box of the diagram commutes up to sign,
and apply that version directly. 

\smallskip
  
If $X$ is compact, then there is no difference between
 $H^k$ and $H^k_c$. In that case, we have

 \begin{cor} [Poincare duality for compact manifolds] If $X$ is a
   compact oriented manifold, then $H^k(X)$ and $H^{n-k}(X)$ are dual
   spaces, with a pairing given by integration.
  \end{cor}
    
\smallskip

\noindent{\bf Exercise 6:} Suppose that $X$ is compact, oriented,
connected and simply connected.
Show that $H^1(X)=0$, and hence that $H^{n-1}(X)=0$. [Side note:
The assumption
of orientation is superfluous, since all simply connected manifolds are
orientable.] 

\smallskip

This exercise shows that a compact, connected, simply connected 3-manifold
has the same cohomology groups as $S^3$. The recently-proved Poincare 
conjecture asserts that all such manifolds are actually homeomorphic to $S^3$.
There are plenty of other (not simply connected!) compact 3-manifolds whose 
cohomology groups are also the same as those $S^3$, namely $H^0=H^3=\R$ and 
$H^1=H^2=0$. These are called {\em rational homology spheres}, and come up
quite a bit in gauge theory. 

\section{Degrees of mappings}
    
Suppose that $X$ and $Y$ are both compact, oriented
$n$-manifolds. Then $H^n(X)$ and $H^n(Y)$ are both naturally
identified with the real numbers. If $f: X \to Y$ is a smooth map,
then the pullback $f_n^\sharp: H^n(Y) \to H^n(X)$ is just
multiplication by a real number. We call this number the {\em degree
  of $f$}. Since homotopic maps induce the same map on cohomology, we
immediately deduce that homotopic maps have the same degree.

However, we already have a definition of degree from intersection
theory! Not surprisingly, the two definitions agree.

\begin{thm} Let $p$ be a regular value of $f$, an let $D$ be the
  number of preimages of $p$, counted with sign.  If $\alpha \in
  \Omega^n(Y)$, then $\int_X f^*\alpha = D \int_Y \alpha$.
\end{thm}

\begin{proof} First suppose that $\alpha$ is a bump form localized in
  such a small neighborhood $V$ of $p$ that the stack-of-records
  theorem applies.  That is, $f^{-1}$ is a discrete collection of sets
  $U_i$ such that $f$ restricted to $U_i$ is a diffeomorphism to $V$.
  But then $\int_{U_i} f^*\alpha = \pm \int_V \alpha = \pm \int_Y\alpha$,
where the $\pm$ is the sign of $\det f$ at the preimage of $p$. 
But then $\int_X f^* \alpha = \sum_i \int_{U_i} f^* \alpha = \sum_i
  \hbox{sign}(\det(df)) \int_V \alpha = D \int_Y \alpha$.

  Now suppose that $\alpha$ is an arbitrary $n$-form. Then $\alpha$ is
  cohomologous to a bump form $\alpha'$ localized around $p$, and
  $f^*\alpha$ is cohomologous to $f^*(\alpha')$, so
$$ \int_X f^*\alpha = \int_X f^*(\alpha') = D \int_Y \alpha' = D \int_Y \alpha.$$
\end{proof}

Here is an example of how this shows up in differential geometry. Let
$X$ be a compact, oriented 2-manifold immersed (or better yet,
embedded) in $\R^3$. At each point $x \in X$, the normal vector $\vec
n(x)$ (with direction chosen so that $\vec n$ followed by an oriented
basis for $T_xX$ gives an oriented basis for $T_x(\R^3)$) is a point
on the unit sphere. That is, $\vec n$ is a map $X \to S^2$. $d \vec n$ is
then a map from $T_x(X)$ to $T_{\vec v(x)}S^n$. But $T_x(X)$ and $T_{\vec v(x)}S^n$
are the same space, being the orthogonal complement of $\vec v(x)$ in 
$\R^{n+1}$. Composing $d \vec v$ 
with this identification of $T_x(X)$ and $T_{\vec v(x)}S^n$, we get an operator
$S: T_x(X) \to T_x(X)$ called the {\em shape operator} or {\em Weingarten
map}. [Note: Some authors use $-d\vec v$ instead of $d \vec v$. This
amounts to just flipping the sign of $\vec v$.] The eigenvalues of $S$ are 
called the {\em principal curvatures} of $X$ at $x$, and the determinant
is called the {\em Gauss curvature}, and is denoted $K(x)$.

Let
$\omega_2$ be the area 2-form on $\S^2$ (explicitly: $\omega_2 = x dy
\wedge dz + y dz \wedge dx + z dx \wedge dy$). The pullback $\vec n^*
\omega_2$ is then $K(x)$ times the area form on $X$. 

\smallskip

\noindent {\bf Exercise 7:} Show that the last sentence is true
{\em regardless of the orientation of $X$}.

\smallskip

The following three exercises are designed to give you some intuition
on what $K$ means.

\noindent{\bf Exercise 8:} Let $X$ be the hyperbolic paraboloid
$z=x^2-y^2$. Show that $K$ at the origin is negative, regardless of
which orientation you pick for $X$. In other words, show that $\vec n$
is orientation reversing near $0$.

\noindent {\bf Exercise 9:} Let $X$ be the elliptic paraboloid
$z=x^2+y^2$. Show that $K$ at the origin is positive.

\noindent {\bf Exercise 10:} Now let $X$ be a general paraboloid $z =
ax^2 + bxy + cy^2$, where $a$ and $b$ and $c$ are real numbers.
Compute $K$ at the origin. [Or for a simpler exercise, do this for
$b=0$ and arbitrary $a$ and $c$.]

\begin{thm}[Gauss-Bonnet] $\int_X K(x) dA = 2 \pi \chi(X)$. \end{thm}

\begin{proof} Since the area of $S^2$ is $4 \pi$, we just have to show
  that the degree of $\vec v$ is half the Euler characteristic of
  $X$. First we vary the immersion to put our Riemann surface of genus
  $g$ in the position used to illustrate the ``hot fudge'' map.  This
  gives a homotopy of the original map $\vec v$, but preserves the
  degree. Then $(0,0,1)$ is a regular value of $\vec n$, and has $g+1$
  preimages, of which one (at the top) is orientation preserving and
  the rest are orientation reversing. Thus the degree is $1-g =
  \chi(X)/2$. \end{proof}

This construction, and this theorem, extends to oriented hypersurfaces 
in higher dimensions.
If $X$ is a compact oriented $n$-manifold in $\R^{n+1}$, we can define the
oriented normal vector $\vec v(x) \in S^n$, so we still have a map
$\vec v: X \to S^n$ and a shape operator $S: T_x(X) \to T_x(X)$. 
The shape operator is always self-adjoint, and so is diagonalizable with 
real eigenvalues, called the principal curvatures of $X$, and orthogonal 
eigenvectors, called the principal {\em directions}. 
Let $\omega_n$ be the volume form on $S^n$, and
we write $\vec v^* \omega_n = K(x) dV$, where $dV$ is the volume form
on $X$. As before, $K(x)$ is called the Gauss curvature of $X$, and is the 
determinant of $S$. 

\begin{thm}[Gauss-Bonnet in even dimensions] If $X$ is a compact $n$-dimensional hypersurface in $\R^{n+1}$, with $n$ even, then $\int_X K(x) dV =
\frac12 \gamma_n \chi(X)$, where $\gamma_n$ is the $n$-dimensional volume
of $S^n$. 
\end{thm}

\begin{proof} As with the 2-dimensional theorem, the key is showing that 
the degree of $\vec v$ is half the Euler characteristic of $X$. Instead of 
deforming the immersion of $X$ into a standard form, we'll use the Hopf 
degree formula. 

Pick a point $\vec a \in S^n$ such that $\vec a$ and $-\vec a$ 
are both regular values of 
$\vec v$. Then the total number of preimages of $\pm \vec a$, 
counted with sign,
is twice the degree of $\vec v$. We now construct a vector field $\vec w(x)$
on $X$, where $w(x)$ is the projection of $\vec a$ onto $T_x(X)$. This is 
zero precisely where the normal vector is $\pm \vec a$. 

\noindent{\bf Exercise 11:} Show that the index of $\vec w$ at such a point is 
the sign of $\det (d \vec v)$. 

By the Hopf degree formula, the Euler characteristic of $X$ is then twice
the degree of $\vec v$. Since $\int_X K(x) dV = \int_X \vec v^* \omega_n
= \frac{\hbox{Degree}}{2} \gamma_n$, the theorem follows.

\end{proof}

If $X$ is odd-dimensional, then $\chi(X)=0$, but $\int_X K(x) dV$ need not
be zero. A simple counter-example is $S^n$ itself. 

Another application of degrees and differential forms comes up in knot 
theory.  If $X$ and $Y$ are
non-intersecting oriented loops in $\R^3$, then there is a map $f(X
\times Y) \to S^2$, $f(x,y) = (y-x)/|y-x|$. The {\em linking number}
of $X$ and $Y$ is the degree of this map. This can be computed in
several ways. One is to pick a point in $S^2$, say (0,0,1), and count
preimages. These are the instances where a point in $Y$ lies directly
above a point in $X$. In other words, the linking number counts
crossings with sign between the knot diagram for $X$ and the knot
diagram for $Y$.  However, it can also be expressed via an integral
formula.

\noindent {\bf Exercise 12:} Supposed that $\gamma_1$ and $\gamma_2$ are
non-intersecting loops in $\R^3$, where each $\gamma_i$ is, strictly
speaking, a map from $S^1$ to $\R^3$. Then $f$ is a map from $S^1
\times S^1$ to $S^2$.  Compute $f^* \omega_2$ in terms of
$\gamma_1(s)$, $\gamma_2(t)$ and their derivatives.  By integrating
this, derive an integral formula for the linking number.

\end{document}